\documentclass[11pt]{amsart}

\usepackage{amsmath}
\usepackage{amssymb}
\usepackage{amsfonts}
\usepackage{amsthm}
\usepackage{enumerate}
\usepackage[mathscr]{eucal}

\usepackage{bbm}
\usepackage{graphics, graphicx}
\usepackage{verbatim}
\usepackage{tikz}

\newcommand{\trn} {|\!|\!|} 
 
\newcommand{\Btrn}  {\Big |\! \Big|\!\Big |}

\newcommand{\ci}[1]{_{{}_{\!\scriptstyle{#1}}}}

\newcommand{\floor}[1]{\lfloor #1 \rfloor }

\newcommand{\Be}{\begin{equation}}
\newcommand{\Ee}{\end{equation}}

\newcommand{\Bea}{\begin{eqnarray}}
\newcommand{\Eea}{\end{eqnarray}}
\newcommand{\Beas}{\begin{eqnarray*}}
\newcommand{\Eeas}{\end{eqnarray*}}
\newcommand{\Benu}{\begin{enumerate}}
\newcommand{\Eenu}{\end{enumerate}}
\newcommand{\Bi}{\begin{itemize}}
\newcommand{\Ei}{\end{itemize}}

\def\intslash{\rlap{\kern  .32em $\mspace {.5mu}\backslash$ }\int}
\def\qsl{{\rlap{\kern  .32em $\mspace {.5mu}\backslash$ }\int_{Q_x}}}

\def\vth{\vartheta}

\def\floor#1{{\lfloor #1 \rfloor }}
\def\emph#1{{\it #1 }}

\def\ga{\gamma}

\def\cf{{\it cf}}

\def\noi{\noindent}

\def\meas{{\text{\rm meas}}}

\def\lc{\lesssim}
\def\gc{\gtrsim}

\def\eps{\varepsilon}

\def\ka{\kappa}
            
\def\la{\lambda}             \def\La{\Lambda}

\def\om{\omega}

\def\fM{{\mathfrak {M}}}
\def\fN{{\mathfrak {N}}}

\def\fQ{{\mathfrak {Q}}}

\def\bbN{{\mathbb {N}}}

\def\bbR{{\mathbb {R}}}

\def\bbZ{{\mathbb {Z}}}

\def\sA{{\mathscr {A}}}

\def\cE{{\mathcal {E}}}
\def\cF{{\mathcal {F}}}

\def\cL{{\mathcal {L}}}

\def\cS{{\mathcal {S}}}

\def\cU{{\mathcal {U}}}
\def\cV{{\mathcal {V}}}

\def\cY{{\mathcal {Y}}}
\def\cZ{{\mathcal {Z}}}

 at 10 true pt

\def\be#1{\begin{equation}\label{ #1}}
\def\endeq{\end{equation}}
\def\endal{\end{align}}
\def\bas{\begin{align*}}
\def\eas{\end{align*}}
\def\bi{\begin{itemize}}
\def\ei{\end{itemize}}

\def\eps{\varepsilon}

\def\emph#1{{\it #1}}
\def\textbf#1{{\bf #1}}

\usepackage{bbm}
\def\bbone{{\mathbbm 1}}

\theoremstyle{plain}
  \newtheorem{theorem}{Theorem}[section]

   \newtheorem{proposition}[theorem]{Proposition}
   
   \newtheorem{lemma}[theorem]{Lemma}
   \newtheorem{corollary}[theorem]{Corollary}

\theoremstyle{remark}
   \newtheorem{remark}[theorem]{Remark}

\theoremstyle{definition}

\begin{document}
\title[Embeddings for  spaces of Lorentz-Sobolev type]
{Embeddings for  spaces of \\ Lorentz-Sobolev type}

\author{Andreas Seeger and Walter Trebels}

\address{Andreas Seeger \\ Department of Mathematics \\ University of Wisconsin \\480 Lincoln Drive\\ Madison, WI, 53706, USA} \email{seeger@math.wisc.edu}

\address{Walter Trebels\\
Fachbereich Mathematik\\ Technische Universit\"at Darmstadt\\ Schlossgartenstr. 7\\
64289 Darmstadt, Germany} \email{trebels@mathematik.tu-darmstadt.de}
\begin{abstract} The purpose of this paper is to 
characterize  all embeddings for    versions of Besov and Triebel-Lizorkin spaces where the underlying Lebesgue space  metric is replaced by a Lorentz space  metric. We include two appendices, one on  the relation between classes of 
endpoint Mikhlin-H\"ormander type 
Fourier multipliers, and one on the constant in  the triangle inequality for  the spaces $L^{p,r} $ when $p<1$.
\end{abstract}
\subjclass[2010]{46E35, 46E30, 42B15, 42B25}
\keywords{Lorentz spaces, Besov spaces, Sobolev spaces, Triebel-Lizorkin spaces, Embeddings, Sobolev embeddings, H\"ormander multiplier theorem}


\thanks{Research supported in part
by National Science Foundation grant DMS 1500162.}

\maketitle

\section{Introduction}
We  consider Lorentz space variants of the classical function space scales of  Sobolev, Besov and Triebel-Lizorkin spaces for  distributions on $\bbR^d$. We use the traditional Fourier analytical definition (\cf. \cite{triebel1}) and work with an inhomogeneous Littlewood-Paley decomposition 
$\{\La_k\}_{k=0}^\infty$ which is defined as follows. Pick a  $C^\infty$ function $\beta_0$ such that $\beta_0(\xi)=1$ for $|\xi|\le 3/2$ and $\beta_0(\xi)=0$ for  $|\xi|\ge 7/4$.
For $k\ge 1$ let $\beta_k(\xi)=\beta_0(2^{-k}\xi)-\beta_0(2^{1-k}\xi)$.  
Define  $\La_k$ via the Fourier transform by $\widehat {\La_k f}=\beta_k\widehat f$, $k=0,1,2,\dots$. 

Let $\cY$ be a rearrangement invariant quasi-Banach space of functions on $\bbR^d$, and 
define
 \begin{align}
&\|f\|_{B^s_q[\cY]}
= \Big(\sum_k 2^{ksq}\big \|\La_k f\|_{\cY}^q\Big)^{1/q},
\\&\|f\|_{F^s_q[\cY]}
= \Big \| \Big(\sum_k |2^{ks} \La_k f|^q\Big )^{1/q}  \Big\|_{\cY}.
\end{align}
When the  functors $B^s_q$ and $F^s_q$ are applied to the Lebesgue spaces 
$\cY=L^p$  one gets the usual classes 
 of Besov spaces $B_q^s[L^p]\equiv B^s_{p,q}$ and Triebel-Lizorkin spaces $F^s_q[L^p] \equiv F^s_{p,q}$.
Here we  take for $\cY$ a Lorentz space $L^{p,r}$, see \S\ref{reviewsect} for  definitions and a review of basic facts. Of course $L^{p,p}=L^p$. It is also  customary to  write $B^s_q[L^{p,r}]=B^s_{(p,r),q}$,  $F^s_q[L^{p,r}]= F^s_{(p,r),q}$, although for better readability we prefer the   functorial notation.
For $q=2$ one obtains the  Lorentz versions of the Hardy-Sobolev spaces, also denoted by
$H^s_{(p,r)}$. 
For the range $1<p<\infty$ the space 
$H^s_{(p,r)} \equiv   F^s_2[L^{p,r}] $ can be identified with a variant of Bessel-potential spaces
(\cf. \cite[ch.V]{stein-sing}), namely we have 
\Be\label{besselpot}\|f\|_{F^s_2[L^{p,r}]}\approx \|(I-\Delta)^{s/2} f\|_{p,r}, \quad 1<p<\infty.\Ee
These spaces have been used repeatedly in the literature (see e.g. \cite{stein-editor}, \cite{cianchi-pick}, \cite{ko-lee}, \cite{grafakos-slavikova}), although our original motivation  came from a result about embeddings in   \cite{gott}. 
Applications suggest   natural questions about the relation between these spaces, in particular the relation between  Besov and Lorentz-Sobolev spaces. 
Our two main theorems 
 characterize all embeddings which involve one  space in the 
 $B^{s}_{q}[L^{p,r}]$ family and one  space in the 
 $F^{s}_{q}[L^{p,r}] $ family.

\begin{theorem}\label{BinFthm}
Let $s_0, s_1\in \bbR$, $0<p_0,p_1<\infty$, $0<q_0, q_1, r_0, r_1\le \infty$. 
The embedding 
\Be
\label{BinFembed}B^{s_0}_{q_0}[L^{p_0, r_0}] \hookrightarrow F^{s_1}_{q_1}[L^{p_1,r_1}]\Ee holds 
if and only if one of the following six conditions is satisfied.

(i) $s_0-s_1>d/p_0-d/p_1>0$.

(ii) $s_0>s_1$, $p_0=p_1$,  $r_0\le r_1$.

(iii) $s_0-s_1=d/p_0-d/p_1>0$, $q_0\le r_1$.

(iv) $s_0=s_1$, $p_0=p_1\neq q_1$, $r_0\le r_1$, $q_0\le \min \{p_1,q_1,  r_1\}$.

(v) $s_0=s_1$, $p_0=p_1=q_1\ge r_0$, $r_0\le r_1$, $q_0\le\min\{ p_1,r_1\}$.

(vi) $s_0=s_1$, $p_0=p_1=q_1<r_0$, $r_0\le r_1$, $q_0< p_1$.
\end{theorem}

\begin{theorem}\label{FinBthm}
Let $s_0, s_1\in \bbR$, $0<p_0,p_1<\infty$, $0<q_0, q_1, r_0, r_1\le \infty$. 
The embedding 
\Be \label{FinBembed}
F^{s_0}_{q_0}[L^{p_0, r_0}] \hookrightarrow B^{s_1}_{q_1}[L^{p_1,r_1}]\Ee holds 
if and only if one of the following six conditions is satisfied.

(i) $s_0-s_1>d/p_0-d/p_1>0$.

(ii) $s_0>s_1$, $p_0=p_1$,  $r_0\le r_1$.

(iii) $s_0-s_1=d/p_0-d/p_1>0$, $r_0\le q_1$.

(iv) $s_0=s_1$, $p_0=p_1\neq q_0$, $r_0\le r_1$, $q_1\ge \max\{p_0,q_0,  r_0\}$.

(v) $s_0=s_1$, $p_0=p_1=q_0\le r_1$, $r_0\le r_1$, $q_1\ge \max\{p_0,r_0\}$

(vi) $s_0=s_1$, $p_0=p_1=q_0>r_1$, $r_0\le r_1$,  $q_1>p_0$.
\end{theorem}


\begin{remark}The  interesting cases deal with the critical relation
\Be \label {sobembrelation}
s_0-s_1=d/p_0-d/p_1,\Ee
when  $p_0<p_1$,  and when $p_0=p_1$. The case 
$p_0<p_1$ in   (iii)  of the two theorems sheds some light on the sharp embedding theorems by Jawerth \cite{jawerth} and Franke \cite{franke}.
Part (iii) of Theorem \ref{FinBthm} extends and improves Jawerth's theorem stating that
$F^{s_0}_{q_0} [L^{p_0}] \hookrightarrow
B^{s_1}_{p_0} [L^{p_1}]$ for any $q_0\le \infty$,  under the assumption \eqref{sobembrelation}, $p_0<p_1$.
Part (iii) of Theorem \ref{BinFthm} extends the dual version of 
 Franke stating that
$B^{s_0}_{p_1} [L^{p_0}] \hookrightarrow F^{s_1}_{q_1}[L^{p_1}],$ for any $q_1>0$, again under the assumption 
 \eqref{sobembrelation}, $p_0<p_1$.  
  For the Hardy-Sobolev case, $q_1=2$,  a partial result 
  of Theorem \ref{BinFthm},  (iii) 
 can be found in \cite{gott}, under the additional assumption  $r_0\le r_1$.
 \end{remark}
 \begin{remark}
  We shall see in Appendix \ref{Hoersect} that an application of  parts (iii) of  Theorems
\ref{BinFthm}   and \ref{FinBthm} 
in tandem is useful 
to compare sharp versions of the H\"ormander multiplier theorem in \cite{seeger1} and \cite{grafakos-slavikova}.
\end{remark}

 Parts (iv), (v), (vi) of both theorems deal with the endpoint case $s_0=s_1$, $p_0=p_1$ in 
 \eqref{sobembrelation}.  The conditions on the $q_i$ and $r_i$ are now more restrictive. The sufficiency of the conditions in (iv), (v), (vi) for 
 \eqref{BinFembed}, \eqref{FinBembed},  resp., 
   follow from corresponding embedding results for 
the spaces 
$\ell^q(L^{p,r})$ and $L^{p,r}(\ell^q)$ 
for sequences of functions $f=\{f_k\}_{k=0}^\infty$.
It turns out that these results can be reduced to two types of triangle inequalities for Lorentz spaces. We note that the two strict inequalities in parts (vi) of both theorems can be traced to the failure of a triangle inequality in $L^{1,\rho}$ for $\rho>1$.  While considering the  results in parts (iv), (v)  of the two theorems 
we came across the question on how  the constants in a generalized triangle inequality  for quasi-norms in $L^{p,\rho}$ depend on $\rho$  when $p<1$ and $p<\rho<\infty$. This dependence is not crucial to our results but  may be interesting in its own right, and we include a result as  Appendix \ref{appendix}.

The above theorems are complemented by more straightforward results about embeddings within the $B^s_q[L^{p,r}]$ and 
$F^s_q[L^{p,r}]$ scales of spaces.

\begin{theorem}\label{BinBthm}
Let $s_0, s_1\in \bbR$, $0<p_0,p_1<\infty$, $0<q_0, q_1, r_0, r_1\le \infty$. 
The embedding 
\Be \label{BinBembed}
B^{s_0}_{q_0}[L^{p_0, r_0}] \hookrightarrow B^{s_1}_{q_1}[L^{p_1,r_1}]\Ee holds 
if and only if one of the following four  conditions is satisfied.

(i) $s_0-s_1>d/p_0-d/p_1>0$.

(ii) $s_0>s_1$, $p_0=p_1$,  $r_0\le r_1$.

(iii) $s_0-s_1=d/p_0-d/p_1>0$, $q_0\le q_1$.

(iv) $s_0=s_1$, $p_0=p_1$,  $r_0\le r_1$, $q_0\le  q_1$.
\end{theorem}

\begin{theorem}\label{FinFthm}
Let $s_0, s_1\in \bbR$, $0<p_0,p_1<\infty$, $0<q_0, q_1, r_0, r_1\le \infty$. 
The embedding 
\Be \label{BinBembed}
F^{s_0}_{q_0}[L^{p_0, r_0}] \hookrightarrow F^{s_1}_{q_1}[L^{p_1,r_1}]\Ee holds 
if and only if one of the following four  conditions is satisfied.

(i) $s_0-s_1>d/p_0-d/p_1>0$.

(ii) $s_0>s_1$, $p_0=p_1$,  $r_0\le r_1$.

(iii) $s_0-s_1=d/p_0-d/p_1>0$, $r_0\le r_1$.

(iv) $s_0=s_1$, $p_0=p_1$,  $r_0\le r_1$, $q_0\le  q_1$.
\end{theorem}

It is noteworthy that 
in  statements (iii) of Theorem \ref{BinBthm}, 
for 
the critical relation  \eqref{sobembrelation} and  $p_0<p_1$ 
the parameters $r_0,r_1$  in  the Besov-Lorentz embeddings 
can be chosen arbitrary. Likewise in 
Theorem \ref{FinFthm}, (iii),  
the parameters $q_0,q_1$ are arbitrary.
This result can be quickly derived from Theorems \ref{BinFthm} and  \ref{FinBthm} (see \S\ref{conclusion}) and extends results by Jawerth \cite{jawerth} for the Lebesgue space cases $p_0=r_0$, $p_1=r_1$.

\medskip

\noi{\it This paper.}
In \S\ref{reviewsect} we shall  review  basic facts on Lorentz spaces and related spaces 
$\ell^q(L^{p,r})$ and $L^{p,r}(\ell^q)$ 
for sequences of functions $f=\{f_k\}_{k=0}^\infty$.
 In \S\ref{meassmoothness} 
we also discuss various examples 
demonstrating the sharpness of the results; see in particular the overview in \S\ref{overview} for a guide where to find the proof of each necessary condition.
In \S\ref{vectvalemb}  we prove  embedding relations between 
$\ell^q(L^{p,r})$ and $L^{p,r}(\ell^q)$, for fixed $p,r$, which imply the sufficiency of the conditions in parts (iv)-(vi) of Theorems \ref{BinFthm} and \ref{FinBthm}.
In \S\ref{jawerthfrankesect}  
we give the proofs of the Lorentz improvements of the embedding theorems by  Franke and Jawerth
(i.e. parts (iii) of Theorems   \ref{BinFthm} and \ref{FinBthm}).
The proofs of sufficiency are concluded in \S\ref{conclusion}.

We  have formulated our results for  the homogeneous 
spaces  $B^s_q(L^{p,r})$, $F^s_q(L^{p,r})$, 
 but 
the proofs can be extended to cover  homogeneous versions 
$\dot B^s_q(L^{p,r})$ and $\dot F^s_q(L^{p,r})$, defined via the usual homogeneous dyadic frequency decompositions, (\cf. \cite{triebelhom}). 
The corresponding results for the  homogeneous variants then hold for the critical 
scale invariant embeddings.

In Appendix \ref{Hoersect} we discuss  some  classes of Fourier multipliers and state some open problems.
In Appendix \ref{appendix} we prove the above mentioned  result on the constant for the triangle inequality 
for $L^{p,r}$ when $p<1$, $r<\infty$.

\medskip

\noi{\it Acknowledgement.} We thank the referee of this paper for a careful reading and for pointing  out possible further directions: First,
the spaces considered here fit in the axiomatic approach by Hedberg and Netrusov \cite{hedberg-netrusov}.
 Secondly,  it might be  of interest to also  replace the sequence spaces  in our definitions by their Lorentz versions and settle analogous embedding questions.

\section{Review of  basic facts on Lorentz spaces}\label{reviewsect}
We review some basic facts about Lorentz spaces, and refer the reader to  \cite{besh}, \cite{BL}, \cite{hunt}, \cite{stw} for more information.
\subsection{\it Lorentz spaces via the distribution function}
Let
$(X, \mu)$ be a measure space.
For a measurable  function $f$ we let 
$$\mu_f(\alpha)= \mu(\{x\in X: |f(x)|>\alpha\}),$$
be the distribution function and 
$f^*(t)=\inf\{\alpha:\mu_f(\alpha)\le t\}$ be the nonincreasing rearrangement of $|f|$.
We shall assume  that $\mu$ is non-atomic (i.e. every set of positive measure has a subset of smaller positive measure).

For $0<p,r<\infty$, the standard quasi-norm on the Lorentz space $L^{p,r}$ is given by
\Be\label{qn}\|f\|_{p,r}= \Big (\frac rp\int [t^{1/p} f^*(t)]^r \frac {dt}t\Big)^{1/r},
\Ee moreover $\|f\|_{p,\infty}=\sup_t t^{1/p} f^*(t)$, $0<p<\infty$.
There is also an alternative description via the distribution function, namely
\Be\label{Lorentznorms}
\|f\|_{p,r}= \Big (r\int [\mu_f(\alpha) ^{1/p} \alpha]^r \frac {d\alpha}\alpha\Big)^{1/r},
\Ee and $\|f\|_{p,\infty}=\sup_\la \la \mu_f(\la)^{1/p}$.
One checks this for simple functions first, and then applies the monotone convergence theorem.
The analogue for the case $r=\infty$  is done in Stein-Weiss
\cite[p.191]{stw}, and the case $r<\infty$, for simple functions relies on similar summation by parts arguments. 

  For later use we state the usual embedding  for fixed $p$, namely   
 $L^{p,r}\hookrightarrow L^{p,q}$ for $r\le q$. In fact there is 
 the sharp inequality 
 \Be \label{embineq}\|f\|_{p,q}\le \|f\|_{p,r}, \quad 0<r\le q\le \infty\,.\Ee
 A proof using rearrangements is in Stein-Weiss \cite[p.192]{stw}, but the proof of \eqref{embineq} could also be based on \eqref{Lorentznorms}, \cf. Lemma \ref{embeddinglemma} in the appendix.

\subsection{\it Sequences of functions}  The study of function spaces crucially relies on the study of the sequence spaces
$L^{p,r}(\ell^q)$ and $\ell^q(L^{p,r})$. We shall work with the 
 quasi-norms
\begin{subequations}\label{seq-norms}
\begin{align}
\|f\|_{L^{p,r}(\ell^q)} &:= \Big\|\Big(\sum_k|f_k|^q\Big)^{1/q}\Big\|_{p,r},
\\
\|f\|_{\ell^q(L^{p,r})} &:= \Big(\sum_k\big\|f_k\big\|_{p,r}^q\Big)^{1/q}.
\end{align}
\end{subequations}
Throughout the paper the domains of the sequences will usually be  $\bbN\cup\{0\}$, or a subset of it,  but it could be  any finite or countable set with counting measure.

 \subsection{\it Powers} 
It will be convenient to use formulas for the distribution and rearrangement functions of $|f|^\sigma$,
for any $\sigma>0$, 
namely 
\Be
\mu_{|f|^\sigma}(\alpha)= \mu_f(\alpha^{1/\sigma})
\Ee
and 
\Be \label {rearrangef^r}
(|f|^\sigma)^*(t)= (f^*(t))^\sigma.
\Ee
These follow directly from the definition of distribution and rearrangement functions.
An immediate consequence is
\Be\label{lor-power}
\| |f|^\sigma \|_{L^{p/\sigma,r/\sigma}}= \|f\|_{L^{p, r}}^{\sigma}.
\Ee
 Moreover, for sequences of functions $f= \{f_k\}$,
 \begin{subequations} \label{powers-for-seqnorms}
 \begin{align}
 \big\|\{ |f_k|^\sigma\}\big\|_{L^{p/\sigma, r/\sigma}(\ell^{q/\sigma})}&= \|f\|^\sigma_{L^{p,r}(\ell^q)},
 \\
   \big\|\{ |f_k|^\sigma\}\big\|_{\ell^{q/\sigma} (L^{p/\sigma,r/\sigma})}
   &=\|f\|_{\ell^q(L^{p,r})}^\sigma\,.
   \end{align}
   \end{subequations}
 
 \subsection{\it Real interpolation} \label{interpolationpara} It is well known that the Lorentz spaces occur as real interpolation spaces for the real method, namely we have
 $(L^{p_0}, L^{p_1})_{\theta,r}  =  L^{p,r}$ if $p_0\neq p_1$, $0<\theta<1$ and  $(1-\theta)p_0^{-1}+\theta p_1^{-1}=p^{-1}$. This remains true for vector-valued spaces (in particular $\ell^q$-valued spaces), see \cite[\S5.2,\S5.6]{BL}. Finally 
 if $\|a\|_{\ell^q_s(A)}= (\sum_{k=0}^\infty  2^{ksq}  \|a_k\|_{A}^q)^{1/q}$ then  
 for vector-valued $\ell^q$ spaces one has
 $(\ell_{s_0}^{q_0}(A_0),\ell_{s_1}^{q_1}(A_1))_{\theta,q}=\ell^q_s((A_0,A_1)_{\theta,q})$ for $(s,q^{-1})=(1-\theta)(s_0,q_0^{-1})+\theta (s_1,q_1^{-1})$, see \cite[5.6.2]{BL}.  Examining e.g. arguments in \cite{triebel2}
 these observations can be used to extend some of the known characterizations of Besov and Triebel-Lizorkin spaces to their Lorentz-space based analogues.

\subsection{\it Sums}
The expression \eqref{qn} is not a norm unless $1\le r\le p$. 
It is well known that the spaces $L^{p,r}$  are normable for $p>1$ and $r\ge 1$;
one replaces $f^*$ by
the maximal function $f^{**}$ in the definition of the Lorentz spaces to get an equivalent expression which is a 
norm.
We write
\[\trn f\trn_{p,r} = \Big(\int_0^\infty t^{r/p} f^{**}(t)^r \frac{dt}{t}\Big)^{1/r}.\]
We also use  $\trn f\trn_{L^{p,r}(\ell^q)}$, $\trn f\trn_{\ell^q(L^{p,r})}$ for the expressions corresponding to \eqref{seq-norms}, but with the ${}^{**}$-functions.
See \cite{hunt} or \cite{besh}. The additivity property holds  when the measure space is nonatomic,  since 
 in these cases we have a triangle inequality for 
 \begin{align} \label{**fct} f\mapsto f^{**}(t)
=\frac 1t\int_0^t f^*(s) ds=
\sup_{E: \mu(E)\le t} \frac{1}t \int_E |f|d\mu\,.
\end{align}
See \cite{hunt} or \cite[ch.2]{besh}.
The true norms can be used to prove duality theorems; one identifies the dual of $L^{p,q}$, $1<p<\infty$, $1\le q<\infty$ with $L^{p',q'}$. This also works  on discrete spaces, with counting measure (see \cite[ch.2.4]{besh}).

If we formulate the triangle inequality with the original quasi-norms in \eqref{qn} we get
for nonatomic $\mu$, 
\Be\label{trianglep>1}
\Big\| \sum_k f_k\Big\|_{{p,r}}\le 
C_{p,r}\sum_k \big\|f_k\big \|_{{p,r}}, \text{ $1<p<\infty$, $r\ge 1$, }
\Ee
with $C_{p,r}=C_p=(1-p^{-1})^{-1}$. This is proved using the additivity of the functional in \eqref{**fct} in combination with Minkowski's and Hardy's inequalities
(\cite[p. 124]{besh}).  Lorentz \cite{lorentz2} showed that one can take $C_{p,r}=1$ for $1\le r\le p$. Barza, Kolyada and Soria \cite{bks} showed for $1<p<r$  that the best constant $C_{p,r}$ in 
\eqref{trianglep>1} is given by
$(p/r)^{1/r} (p'/r')^{1/r'}$.

One  can use \eqref{powers-for-seqnorms} and 
\eqref{trianglep>1} for the space $L^{p/u, r/u}$ to get
\Be\label{utrianglep>1}
\Big\| \sum_k f_k\Big\|_{{p,r}}\lc_{p,r,u} 
\Big(\sum_k \big\|f_k\big \|_{{p,r}}^u\Big)^{1/u}, \text{ 
$u<\min \{p,r,1\}$.}
\Ee
However this  can be improved in some cases.
The analogue of \eqref{trianglep>1} {\it fails} for $L^{1,r}$, $r>1$, (\cf. \cite{st-nw} for a weaker substitute) but there is  a different kind of triangle inequality for 
$p<1$, for the $p$th power of $\|\cdot\|_{p,r}$, which gives
\Be \label{ptrianglelem}  
\Big\|\sum_k f_k \Big\|_{{p,r}} \le C(p,r) \Big(\sum_k \|f_k\|_{{p,r}}^p \Big)^{1/p}, 
\quad p<1,\quad r\ge p;
\Ee here $C(p,r)\le  (\frac{2-p}{1-p})^{1/p}$.
This was proved 
for $r=\infty$  by Stein, Taibleson and Weiss  
\cite{sttw},  (see also  Kalton  \cite{kalton}, and  unpublished work of Pisier and Zinn mentioned in \cite{kalton}).  
It is easy to modify the proof in \cite{sttw} to cover the cases $p<r<\infty$ with the same  constant.
However for  $r=p$ one  can put of course put $C(p,p)=1$ which suggests that the behavior of $C(p,r)$ should improve for $r>p$ as $r$ decreases.
We shall prove such a result in Appendix \ref{appendix} and show that for $0<p<1$, $p<r<\infty$
\Be \label{Cpr} C(p,r) \le A^{1/p}  \Big( \frac 1{1-p}\Big )^{1/p-1/r} \Big(1 + \frac pr \log \frac{1}{1-p} \Big)^{1/p-1/r}
\Ee
and $A$ does not depend on $p$ and $r$.
 The precise behavior of $C(p,r)$ is not relevant for the results in this paper, but \eqref{Cpr}   should be 
 interesting in its own right. 
Note that the logarithmic term in \eqref{Cpr} vanishes as $r\to p+$ and as $r\to \infty$. 

\smallskip
{\it Open problem.} It would be interesting to get more precise information on $C(p,r)$, in particular one would like to know 
whether the logarithmic term in \eqref{Cpr} is necessary  for $p<r<\infty$.

\subsubsection{Computations of some lower bounds}\label{Lprcomputationslower}
Suppose  we are given  $b>0$ and sets $A_j$, indexed by $j\in \cZ\subset \bbZ$ such that
\begin{subequations}\Be\mu(A_j)\ge b \rho^j, \,\,\, j\in \cZ\Ee
for some $\rho>1$. Assume that, for a nonnegative sequence $\beta\in 
\ell^r(\bbZ)$,
\Be\label{lowerbdassuLpr}
|f(x)|\ge \sum_{j\in \cZ}\beta(j) \rho^{-j/p} \bbone_{A_j}(x)\,\, \text{ a.e.}\Ee
Then for $0<p<\infty$, $0<r\le\infty$,
\Be\label{Lprlowerbdconcl} \|f\|_{p,r}\gc  \Big(\sum_{j\in \cZ}|\beta(j)|^r\Big)^{1/r}\Ee
\end{subequations}
with the obvious sup-analogue for $r=\infty$. The implicit constants depend on $p,r$.

To prove the lower bound  observe that the distribution function satisfies 
$\mu_f(\beta(j) \rho^{-\frac{j+1}{p}})\ge \mu (A_j) > b\rho^{j-1}$ and therefore
$f^*(b\rho^{j-1}) \ge \beta(j)\rho^{-\frac{j+1}{p}}$ by definition of the rearrangement function.
This easily implies \eqref{Lprlowerbdconcl}, 
under the assumption \eqref{lowerbdassuLpr}.

\subsubsection{Computations of some upper bounds}\label{Lprcomputationsupper}
We now  replace $\cZ$ by $\bbZ$  and add the assumption
that  {\it $n\mapsto \beta(n) 2^{-n/p}$ is nonincreasing}. Assume that $\{F_n\}_{n\in \bbZ}$ is a sequence of measurable sets such that 
\begin{subequations}
\Be\mu(F_n)\le B \rho^n, \,\,\, n\in \bbZ,\Ee
for some $\rho>1$, $B>0$
and assume that, 
\Be |f(x)|\le \sum_{n\in \bbZ}  \beta(n) 2^{-n/p}\bbone_{F_n}(x)\,\, \text{ a.e.}\Ee
Then for $0<p<\infty$, $0<r\le \infty$,
\Be \label{upperbdassuLpr}
\|f\|_{p,r} \lc  \big\|\beta\|_{\ell^r(\bbZ)}.\Ee
\end{subequations}
To see this observe that $$\mu_f(\beta(n) 2^{-\frac{n-1}p})\le \mu(\cup_{j\le n} F_j) 
\le B\sum_{j\le n}\rho^j \le B(\rho-1)^{-1}\rho^{n+1}$$
and therefore
$ f^*(\frac{B}{\rho-1}\rho^{n+1}) \le \beta(n) \rho^{-\frac{n-1}p}$.
This easily implies \eqref{upperbdassuLpr}.

\section{Necessary conditions}
\label{meassmoothness}
  
  \subsection{\it Guide through this section} \label{overview}
  Many examples for embedding relations of spaces of Hardy-Sobolev type  have been discussed in the literature (e.g  \cite{taibleson}, \cite{sickel-tr}), and most of our examples are at least related to those earlier examples.
  
  The necessity of the condition $p_0\le p_1$, and the 
  necessity of the condition  $r_0\le r_1$ in the case 
        whenever $p_0=p_1$, in all four Theorems
  \ref{BinFthm},
    \ref{FinBthm},
      \ref{BinBthm},
        \ref{FinFthm},  is proved in \S\ref{p0lep1}. 
        The necessity of the condition $s_0-s_1\ge d/p_0-d/p_1\ge 0$ in all four theorems is proved in \S\ref{scalingexample}.

       Consider the case $s_0-s_1=d/p_0-d/p_1>0$ which is case (iii) in all four theorems.
The necessity of the condition $q_0\le r_1$ in Theorem \ref{BinFthm} (iii),
the necessity of the condition $r_0\le q_1$ in Theorem \ref{FinBthm} (iii),
the necessity of the condition $r_0\le r_1$ in Theorem \ref{BinBthm} (iii), and 
  the necessity of the condition $q_0\le q_1$ in Theorem \ref{FinFthm} (iii), are all proved in \S\ref{franke-jawerth-nec}.
  
 \subsubsection*{Necessary conditions in the case  $s_0=s_1$ and $p_0=p_1$} 
  In Theorem \ref{BinFthm} (iv), (v)
   the necessity of the condition $q_0\le p_1$ is shown in 
  \S\ref{spequality-nec-pq},
   the necessity of the condition $q_0\le q_1$ is shown in 
  \S\ref{q0q1necsect},
  and
   the necessity of the condition $q_0\le r_1$ is shown in 
  \S\ref{franke-jawerth-nec}.  In addition, for the  case $p_1=q_1<r_0$ in part (vi), we must have the strict inequality in $q_0<p_1$; this follows from \eqref{cexbf4} in \S\ref{spequality-nec-pr}.

In Theorem \ref{FinBthm} (iv), (v)
   the necessity of the condition $q_1\ge p_0$ is shown in 
  \S\ref{spequality-nec-pq},
   the necessity of the condition $q_0\ge q_1$ is shown in 
  \S\ref{q0q1necsect},
  and
   the necessity of the condition $q_1\ge r_0$ is shown in 
  \S\ref{franke-jawerth-nec}.  Moreover, for the  case $p_0=q_0>r_1$ in part (vi), we must have the strict inequality in $q_1>p_0$; this follows from \eqref{cexbf8} in \S\ref{spequality-nec-pr}.

  The necessity of the conditions $r_0\le r_1$ in Theorems \ref{BinBthm}, (iv), and  \ref{FinFthm},
  (iv),   is shown in \S\ref{p0lep1} (as already pointed out).
  The necessity of the conditions $q_0\le q_1$ in Theorems \ref{BinBthm}, (iv), and \ref{FinFthm},
  (iv),   is shown in \S\ref{q0q1necsect}.

\subsection{\it Preliminaries}\label{prelpsi}
In what follows we let $\psi_0$ be a $C^\infty$ function supported on $\{x:|x|\le 1/8\}$  such that
$\widehat \psi_0(\xi)\neq 0$ for $|\xi|\le 2$ and such that 
\Be\label{momentcond}\begin{aligned} 
\int \psi_0(x) dx=1,\quad &\int \psi_0(x)x_1^{\nu_1}\dots x_d^{\nu_d} dx =0, \\
&\text{ for multiindices $\nu$ with $\nu_i\ge 0$, $0<\sum_{i=1}^d {\nu_i}\le M$}.
\end{aligned}\Ee
Here we assume that  $M$  is large, specifically given $p,q,s $, the condition  $$M>|s|+ 100 d \max\{1, 1/p, 1/q\}$$ will certainly be sufficient for our purposes.
  Let \begin{subequations}\label{Lkdef}\begin{gather}\psi_k = 2^{kd}\psi_0(2^k\cdot)- 2^{(k-1)d}\psi_0(2^{k-1}\cdot), \quad k\ge 1,
  \\
  \cL_k f=\psi_k*f
  \end{gather}
  \end{subequations}
We can arrange $\psi_0$ so that $\psi_1$ satisfies
\Be\label{lowerbdonpsipsi}
\psi_1*\psi_1(x)\ge c \text { for } x\in I=(-\eps,\eps)^d
\Ee
for some fixed $\eps>0$.
We have (using Littlewood-Paley decompositions generated by dilates of compactly supported functions, aka  local means in \cite{triebel2}) 
\begin{subequations}\label{localmeanschar}
\begin{align}
&\|f\|_{B^s_q[L^{p,r}]}
\approx  \Big(\sum_{k=0}^\infty 2^{ksq}\|\cL_k f\|_{p,r}^q\Big)^{1/q},
\\&\|f\|_{ F^s_q[L^{p,r}]}\approx \Big \|\Big(\sum_{k=0}^\infty 2^{ksq} |\cL_k f|^q\Big)^{1/q}
\Big \|_{p,r}.
\end{align}
\end{subequations}
These equivalences  follow by modifying the  corresponding arguments  for the case  $p=r$ (\cf. \cite{triebel2}) and combining them with real interpolation arguments, cf. \S\ref{interpolationpara}. We omit the details.

We will repeatedly need the following straightforward lemma.
\begin{lemma} \label{momentlemma}
Let $W$ be a finite collection of  points in $\bbR^d$, with mutual distance at least $2^{-n}$. Define
$h(x)=\sum_{w\in W} \psi(2^n(x-w))$. Then 
 \Be\label{ptwLjh}
 \|\cL_j h\|_\infty \lc 
\begin{cases} 2^{(n-j)(M+1)} &\text{ if $j\ge n$},
\\ 2^{(j-n)(M+1-d)} &\text{ if $n\ge j$}.
\end{cases} \Ee
Moreover,
\begin{subequations} \Be\|\cL_j h\|_{p,r}\lc 2^{(n-j)M} (2^{-nd}\,\# W)^{1/p}, \quad \text{if $j\ge n$}.\Ee
and for $j=n$ we have the equivalence
\Be\|\cL_n h\|_{p,r} \approx (2^{-nd}\# W)^{1/p}.\Ee
For $j\le n$ let $W(j)$ be any maximal $2^{-j}$-separated subset of $W$. Then
\Be\|\cL_j h\|_{p,r}\lc 2^{(j-n) (M-d)} (2^{-jd}\,\# W(j))^{1/p},
\quad \text{if $j\le n$}.\Ee
\end{subequations}
\end{lemma}
\begin{proof}
Note that the upper bounds  need to be proved only for $p=r$ as they follow then for all $r$ by real interpolation. Let $h_w=\psi(2^n(\cdot-w))$. 
The derivation is straightforward; we use the moment condition
on the convolution kernel $\psi_j$ for the operator $\cL_j$ to bound
$|\cL_j h_w(x)|\lc 2^{(n-j)(M+1)}$ for $j> n$. The corresponding $L^p$ bound follows as the supports of $\cL_j h_w$ are essentially disjoint.
When $j\le n$ we use the moment condition on $h_w$ to bound
$|\cL_j h_w(x)|\lc 2^{(j-n)(M+1)}$.  The bound for $\cL_j h$  follows since  $\cL_j h_w(x)$  is nonzero for at most $O(2^{(n-j)d})$ terms
(and one gets improvements for sparse $W$).
The corresponding $L^p$ bound is then an immediate consequence.

In order to obtain  the lower bound for $\|\cL_nh\|_{p,r}$ we use the assumption \eqref{lowerbdonpsipsi} to see that $|\cL_n h|\ge c$  on a set of measure $2^{-nd}\,\# W$.
\end{proof}

In what follows we shall denote by $B(x,\rho)$ the ball of radius $\rho$ centered at $x$.

  \subsection{\it Necessity of  $p_0\le p_1$,  and of $r_0\le r_1 $ in the case  $p_0=p_1$.}\label{p0lep1}

  Let $\psi_1 $ be as in \S\ref{prelpsi}, $e_1=(1,0,\dots,0\}$, and let $$f(x) = \sum_{n=1}^\infty a_n \psi_1(x-ne_1)\,.$$ 
  It is easy to see by \eqref{ptwLjh} and  Minkowski's inequality that for any $M$
  $$|\cL_k f(x)| \lc 2^{-kM} \sum_{n=1}^\infty |a_n| \bbone_{B(ne_1,1)}(x).$$
  This implies  $\|\cL_k f\|_{p,r}\lc 2^{-kM} \|a\|_{\ell^p}$.
  We also have
   \begin{align*} |\cL_1 f(x)| = \Big|
   \sum_{n=1}^\infty a_n \psi_1*\psi_1(x-ne_1)\Big|
   \ge 
   \sum_{n=1}^\infty |a_n| \bbone_{B(ne_1,\eps)}(x) 
   \end{align*} which implies the lower bound 
    $\|\cL_1 f\|_{p,r} \gc \|a\|_{\ell^{p,r}}$. 
    It follows that $p_0\le p_1$ in all cases.

The same calculation
proves the necessary condition $r_0\le r_1$ in the case $p_0=p_1$.

\subsection{\it Necessity of  $s_0-s_1\ge d(1/p_0-1/p_1)$.} \label{scalingexample}
Let $\chi\in C^\infty_c(\bbR^d)$ be supported in the ball of radius $10^{-2}$ centered at $1$.  Let $\beta_k$ be as in the definition of $\La_k$ in the introduction, so that
for $k\ge 1$ we have 
$\beta_k(\xi)=0$ when $2^{-k}|\xi|\notin [\frac 34, \frac 74]$ and 
$\beta_k(\xi)=1$ when $2^{-k}|\xi|\in [\frac 78, \frac 32]$. 
Let $\om_k= 2^{kd}\cF^{-1}[\chi](2^k x) $ and notice that
$\La_k\om_k=\om_k$ and $\La_j\om_k=0$ when $j\neq k$.
We have by scaling  $ \|\La_k\om_k\|_{p,r} = 2^{k (d-d/p )} \|\cF^{-1}[\chi]\| _{p,r}$ and thus any of the embeddings  in the four theorems in the introduction requires
$2^{k(s_1-d/p_1)}\lc 
2^{k(s_0-d/p_0)}$ 
for $k\ge 0$. Hence $s_0-s_1\ge d(1/p_0-1/p_1)$.

\subsection{\it Necessary conditions for the case $s_0-s_1= d(1/p_0-1/p_1)\ge 0$.}
\label{franke-jawerth-nec}


Let $R\gg 8$ be large and  let  $\{ n_l\}_{l=1}^\infty$  be an increasing  sequence of integers which is sufficiently separated, i.e. such that   $n_l\gg l\ge R$, $n_{l+1}-n_l\ge  R$.
 Let $\fN:=\{n_l: l\in \bbN\}$.
 Let  $\{a_l\}_{l=1}^\infty $ be a decreasing sequence for which $l\mapsto 2^{n_ld/p}|a_l|$ is increasing. Define
 $\Psi_{n}(x):=\psi_1(2^n(x-2^{-n} e_1))$, with $\psi_1$ as in \S\ref{prelpsi} and 
\Be\label{hgamma}h_\gamma(x)=\sum_{l=1}^\infty a_l 2^{n_l\gamma}\Psi_{n_l}(x).
\Ee
\begin{lemma}\label{secondexamplelemma}
Let $s\in \bbR$. If the separation constant $R$ 
in the definition of $\fN$ is 
sufficiently large then  
\begin{align}
\label{hgammaB}
\big\|h_{-s+d/p}\big\|_{B^s_q[L^{p,r_0}]} &\approx \|a\|_{\ell^q},
\\
\label{hgammaF}
\big\|h_{-s+d/p}\big\|_{F^s_\rho[L^{p,r}]} &\approx \|a\|_{\ell^r}.\, 
\end{align}
\end{lemma} 
As an immediate consequence we get
\begin{corollary} Suppose $s_0-s_1=d/p_0-d/p_1\ge 0$. Then we have the following implications
\begin{subequations}\begin{alignat}{2}
\notag
&B^{s_0}_{q_0}[L^{p_0,r_0}] \hookrightarrow B^{s_1}_{q_1}[L^{p_1,r_1}] &&\implies q_0\le q_1,
\\ \notag
&F^{s_0}_{q_0}[L^{p_0,r_0}] \hookrightarrow F^{s_1}_{q_1}[L^{p_1,r_1}] &&\implies r_0\le r_1,
\\ \notag
&B^{s_0}_{q_0}[L^{p_0,r_0}] \hookrightarrow F^{s_1}_{q_1}[L^{p_1,r_1}] &&\implies q_0\le r_1,
\\ \notag
&F^{s_0}_{q_0}[L^{p_0,r_0}] \hookrightarrow B^{s_1}_{q_1}[L^{p_1,r_1}] &&\implies r_0\le q_1.
\end{alignat}
\end{subequations}
\end{corollary}

\begin{proof}[Proof of Lemma \ref{secondexamplelemma}]
Let $u<\min \{p,q,r_0\}$. By  \eqref{utrianglep>1} we have
\[\|h_\ga\|_{B^s_q[L^{p,r_0}]}\lc 
\Big(\sum_{j=0}^\infty 2^{jsq}\Big(\sum_{l=1}^\infty
\| a_l 2^{n_l\ga } 
\cL_j \Psi_{n_l}\Big\|_{p,r_0}^u\Big)^{q/u}\Big)^{1/q}.
\]
We use Lemma \ref{momentlemma} for a singleton $W$ to estimate 
for $\gamma=-s+d/p$, the right hand side in the last display by a constant times
\begin{align*}
& \Big(\sum_{j=0}^\infty 2^{jsq}\Big(\sum_{l=1}^\infty |a_l|^u 2^{n_l(\ga-d/p) u} 2^{-|j-n_l|(M-2d)u} \Big)^{q/u}\Big)^{1/q}
\\
&\lc \Big(\sum_{j=0}^\infty \Big(\sum_{l=1}^\infty |a_l|^u  2^{-|j-n_l|(M-2d-|s|)u} \Big)^{q/u}\Big)^{1/q}
\\&\lc
\Big(\sum_{j=0}^\infty \sum_{l=1}^\infty |a_l|^q 2^{-|j-n_l|q}\big)^{1/q}
\lc \Big(\sum_{l=1}^\infty |a_l|^q\Big)^{1/q};
\end{align*}
in this calculation  we have used $M-2d-|s|>1$.
We have proved  the upper bound in \eqref{hgammaB}.

For the lower bound we estimate
$$\|h_{-s+d/p}\|_{B^s_q[L^{p,r_0}]}\gc 
\Big(\sum_{k=1}^\infty 2^{n_k sq} \big\|\cL_{n_k} h_{-s+d/p}\|_{p,r_0}^q\Big)^{1/q}\ge
cI-CII$$ where
\begin{align*}
I&=\Big(\sum_{k=1}^\infty 
|a_k|^q 2^{n_k d/p} 
\|\cL_{n_k} \Psi_{n_k} \|_{p,r_0}^q \Big)^{1/q},
\\
II&=\Big(\sum_{k=1}^\infty 
2^{n_k sq} \Big\|\sum_{\substack {l\ge 1: \\l\neq k}} a_l 2^{- n_l s} 2^{n_l d/p} \cL_{n_k} \Psi_{n_l} \Big\|_{p,r_0}^q \Big)^{1/q}.
\end{align*} 
We have (using \eqref{lowerbdonpsipsi})
$2^{n_k d/p} 
\|\cL_{n_k} \Psi_{n_k} \|_{p,r_0}\ge c>0$ uniformly in $k$ and therefore 
$I\gc \|a\|_{\ell^q}$. The above computation for the upper bound also gives
 $II\lc 2^{-R} \|a\|_{\ell^q}$ and the lower bound in \eqref{hgammaB} follows if $R$ is chosen sufficiently large.
 
We now turn to the proof of \eqref{hgammaF}. For the upper bound we may assume without loss of generality  that $\rho<\min \{1,r,p\}$.
Then by Lemma \ref{momentlemma} 
\[\Big(\sum_{j=0}^\infty\Big|2^{js}\sum_{l=1}^\infty a_l 2^{n_l(\frac dp -s)} \cL_j\Psi_{n_l}(x) \Big|^\rho\Big)^{1/\rho}\le \cE_1(x)+\cE_2(x)\]
where
\begin{align*} 
&\cE_1(x)=\Big(\sum_{j=0}^\infty2^{js\rho} \sum_{\substack{l\in \bbN \\n_l\le j} }|a_l|^\rho  2^{n_l(\frac dp-s)\rho} 2^{-(j-n_l) (M+1)\rho} 
\bbone_{B(0,2^{1-n_l})} (x)\Big)^{1/\rho},
\\&\cE_2(x)= 
\Big(\sum_{j=0}^\infty 2^{js\rho} \sum_{\substack{l\in \bbN \\n_l> j}} |a_l|^\rho  2^{n_l(\frac dp-s)\rho} 2^{-(n_l-j)(M+1-d)\rho} 
\bbone_{B(0, 2^{1-j})} (x)\Big)^{1/\rho}.
\end{align*}

Interchanging the $n_l,j$ summations and summing a geometric series (where $M+1>|s|$) yields
\[ \cE_1(x) \le \Big(\sum_{l=1}^\infty |a_l|^\rho 2^{n_l \frac dp \rho} \bbone_{B(0,2^{1-n_l})}(x)\Big)^{1/p}\]
and, with the parameter $m=n_l-j$,
\Be \label{Etwo} \cE_2(x)\le \Big(\sum_{m=0}^\infty 2^{-m(M+1-d+s-\frac dp)\rho}
 \cE_{2,m} (x)^\rho\Big)^{1/\rho}\Ee
where
\[\cE_{2,m} (x) =\Big(\sum_{\substack {l\in \bbN:\\ n_l\ge m} }|a_l|^\rho 2^{(n_l-m) \frac dp \rho} \bbone_{B(0, 2^{1-n_l+m})}(x)  \Big)^{1/\rho}.\]
We use  \eqref{lor-power},  and \S\ref{Lprcomputationsupper} 
 with the parameter $\rho=2^{-d}$ and  with exponents $(p/\rho,r/\rho)$ in place of $(p,r)$, 
 to get
\[\|\cE_1\|_{p,r} =\|\cE_1^{\rho} \|_{p/\rho, r/\rho}^{1/\rho} \lc 
 \|a\|_{\ell^r}.
\]
Similarly  $\|\cE_{2,m}\|_{p,r}\lc \|a\|_{\ell^r} $ uniformly in $m$, and then  from \eqref{Etwo} we also get
$\|\cE_2\|_{p,r} \lc \|a\|_{\ell^r} $.

For the lower bound 
\begin{align*}&\Big(\sum_{j=0}^\infty\Big|2^{js}\sum_{l=1}^\infty a_l 2^{n_l(\frac dp -s)} \cL_j\Psi_{n_l} (x)\Big|^\rho\Big)^{1/\rho}
\\&\ge 
\Big(\sum_{k=1}^\infty\Big|2^{n_ks}\sum_{l=1}^\infty a_l 2^{n_l(\frac dp -s)} \cL_{n_k}\Psi_{n_l}(x) \Big|^\rho\Big)^{1/\rho}
\, \ge \cE_3(x)-\cE_4(x)\end{align*}
where
\begin{align*}
\cE_3(x)&= 
\Big(\sum_{k=1}^\infty\Big|a_k 2^{n_k\frac dp } \cL_{n_k}\Psi_{n_k}(x) \Big|^\rho\Big)^{1/\rho},
\\
\cE_4(x)&=
\Big(\sum_{k=1}^\infty\Big|2^{n_ks}\sum_{\substack{l\in \bbN:\\l\neq k}}  a_l 2^{n_l(\frac dp -s)} \cL_{n_k}\Psi_{n_l}(x) \Big|^\rho\Big)^{1/\rho}.
\end{align*}

Note that $ \cL_{n_k}\Psi_{n_k} (x)
=\psi_1*\psi_1(2^{n_k} x-e_1)$  and thus
$|\cL_{n_k}\Psi_{n_k} (x)| \ge c>0$ on 
$B(2^{-n_k}e_1, 2^{-n_k}\eps)$. Hence
\[ \cE_3(x) \ge c 
\Big(\sum_{k=1}^\infty\Big|a_k 2^{n_k\frac dp } 
\bbone_{B(2^{-n_k}e_1, 2^{-n_k}\eps)}(x)  \Big|^\rho\Big)^{1/\rho},
\]
which by \S\ref{Lprcomputationslower} implies $\|\cE_3\|_{p,r}\gc \|a\|_{\ell^r}$. Analyzing the proof of the upper bound and taking into account the  $R$-separation of the numbers $n_l$  yields $\|\cE_4\|_{p,r}\lc 2^{-R} \|a\|_{\ell^r}
$ and for large $R$ we get the lower bound in \eqref{hgammaF}.
\end{proof}

\subsection{\it The case $s_0=s_1=s$, $p_0=p_1=p$.}
\label{spequality-nec}

\subsubsection{Necessary conditions on $(p,q_0) $ and $(p, q_1)$ in Theorems \ref{BinFthm}, \ref{FinBthm}} \label{spequality-nec-pq}
Let $R\gg 8$ be large and  $\fN=\{n_l: l\in \bbN\}$ be as in \S\ref{franke-jawerth-nec},
i.e.  $l\mapsto n_l$ is increasing and $R$-separated.
\begin{lemma} \label{firstexamplelemma}
Let $a=\{a\}_{l=1}^\infty$ be a sequence such that  $l\mapsto |a_l|$ is nonincreasing. 
Let, for $\nu\in \bbZ^d$,
 $g_{l,\nu}(x) = 
 \psi(2^{n_l}(x-(l e_1+ 2^{-n_l}\nu)))\,$  and  let 
 \[g_{l} (x)=\sum_{\substack{2^{3}\le\nu_i\le 2^{n_l-3}\\ i=1,\dots,d} }g_{l,\nu}(x).\]
  Define
$f(x)=\sum_{l=1}^\infty   2^{- s n_l} a_l g_{l}(x).
$ Then 
\begin{align}
\label{Bupperboundforf}
\|f\|_{B^s_q[L^{p,r}] }&\approx \Big(\sum_{l=1}^\infty 
|a_l|^q\Big)^{1/q},
\\
\label{lprseqequivalence} 
\|f\|_{F^s_q[L^{p,r}]} &\approx 
\Big(\sum_{l=1}^\infty l^{\frac rp -1 }|a_l|^r \Big)^{1/r},
\end{align}
provided that the separation constant $R$ is large enough.
\end{lemma}
Note that the expression of the right hand side of \eqref{lprseqequivalence} is equivalent to the $\ell^{p,r}$ norm of $a$ (if $a$ is nonincreasing).

The lemma  implies  the following corollary relevant for Theorems \ref{BinFthm} and \ref{FinBthm}.
\begin{corollary}
\begin{alignat} {2}
\notag
&B^s_{q_0}[L^{p,r_0}] \hookrightarrow F^s_{q_1} [L^{p,r_1}] \, &&\implies q_0\le  p.
\\
\notag
&F^s_{q_0}[L^{p,r_0}] \hookrightarrow  B^s_{q_1}[L^{p,r_1}] &&\implies p\le q_1.
\end{alignat}
\end{corollary}

\begin{proof}[Proof of Lemma \ref{firstexamplelemma}]
By Lemma \ref{momentlemma} 
\begin{subequations}
\begin{align} 
\label{boundgnlupper}
\|\cL_j g_{l}\|_{{p,r}} &\lc 2^{-|n_l-j| (M-d)},
\\
\label{boundgnlequiv}
\|\cL_{n_l} g_{l}\|_{{p,r}} &\approx 1.
\end{align}
\end{subequations}
We first establish the upper bound in \eqref{Bupperboundforf} and 
let 
 $u<\min\{p,q,r\}$. Estimate using \eqref{utrianglep>1}, \eqref{boundgnlupper} and then 
 H\"older's inequality 
\begin{align*}
\|f\|_{B^s_q[L^{p,r}] }
&\lc 
 \Big(\sum_{j=0}^\infty 2^{jsq}\Big(\sum_{l=1}^\infty
 |a_{l}|^u    2^{-sn_lu } \big\|\cL_j g_{l} \big\|_{{p,r}}^u
 \Big)^{q/u}
\Big)^{1/q} 
\\
&\lc \Big(\sum_{j=0}^\infty \Big(\sum_{l=1}^\infty |a_l|^u 2^{-(M-d-|s|) |j-n_l|u}\Big)^{q/u}\Big)^{1/q}
\\
&\lc 
 \Big(\sum_{j=0}^\infty \sum_{l=1} ^\infty
 |a_{l}|^q   2^{-|j-n_l| q} 
\Big)^{1/q} \lc 
\Big(\sum_{l=1} ^\infty
 |a_{l}|^q  
\Big)^{1/q} 
\end{align*}
where we used $M-d-|s|>1$.

For the lower bound in \eqref{Bupperboundforf} we have
\begin{align*}
\|f\|_{B^s_q[L^{p,r}] }
&\gc  \Big(\sum_{k=1}^\infty 2^{n_ksq} \Big\| a_k 2^{-s n_k} \cL_{n_k} g_{n_k} + \sum_{\substack{l\in \bbN:\\ l\neq k} }a_l  2^{-sn_l } \cL_{n_k} g_{n_l}\Big\|_{{p,r}}^q
\Big)^{1/q}\\&\ge c_I I- C_{II}  II
\end{align*}
where
\begin{align*}
I&= \Big(\sum_{k=1}^\infty |a_k|^q \|\cL_{n_k}g_{n_k}\|_{p,r}^q\Big)^{1/q},
\\
II&\le \Big(\sum_{k=1}^\infty 2^{n_ksq} \Big\|\sum_{\substack{l\in \bbN:\\ l\neq k} }a_l  2^{-sn_l } \cL_{n_k} g_{n_l}\Big\|_{{p,r}}^q
\Big)^{1/q}.
\end{align*}
By \eqref{boundgnlequiv} we get $I\gc\|a\|_{\ell^q}$. Using the argument for the 
upper bound given above and taking account \eqref{boundgnlupper} with $j=n_k\neq n_l$ yields $II\lc 2^{-R} \|a\|_{\ell^q}$. Thus if $R$ is large enough we obtain the lower bound in 
\eqref{Bupperboundforf}.

We now prove  the upper bound in \eqref{lprseqequivalence}.
By the $L^\infty$ bounds in Lemma \ref{momentlemma} we have
$$\Big(\sum_{j=0}^\infty 2^{-jsq}|\cL_j f(x)|^q\Big)^{1/q} \le C G(x)$$
where $G(x)=\sum_{l=1}^\infty |a_l| \bbone_{Q_l}(x)$,  with $Q_l = l+[-1/4,1/4]^d.
$
The rearrangement of $G$ satisfies
$$G^*(t) \le \sum_{l=1}^\infty |a_l|\bbone_{((l-1)2^{-d}, l 2^{-d}]}(t) $$ and we obtain
\begin{align*}
\|f\|_{F^s_q[L^{p,r}]}\lc \Big(\frac rp\int_0^\infty [t^{1/p}G^*(t)]^r\frac {dt}{t}\Big)^{1/r} \lc 
\Big(\sum_{l=1}^\infty l^{\frac rp-1}|a_l|^r\Big)^{1/r}.
\end{align*}

For the lower bound we estimate 
\begin{align*}
\|f\|_{F^s_q[L^{(p,r)}] }&\ge \Big\| \Big(\sum_{l} 2^{n_l sq} |\cL_{n_l} f|^q\Big)^{1/q} \Big\|_{L^{p,r}}
\,\ge cI' -C II'
\end{align*}
where 
\begin{align*}
I'&=  \Big\| \Big(\sum_{l=1}^N \big |a_l  \cL_{n_l} g_{n_l}\big |^q\Big)^{1/q} 
\Big\|_{p,r},
\\
II'&=  \Big\| \Big(\sum_{l=1}^N 2^{n_l sq}\Big | \sum_{\substack{k\ge 1:\\k\neq l}} a_k 2^{-sn_k  } \cL_{n_l} g_{n_k}\Big |^q\Big)^{1/q} \Big\|_{p,r}.
\end{align*}

Let $\eps$ be as in
\eqref{lowerbdonpsipsi} and
$J_{l,\nu,\eps}=
\{x:|x-le_1-\nu 2^{-n_l}|\le 2^{-n_l}\eps\}$. Let $J_{l,\eps}$ be the union of the $J_{l,\nu,\eps}$ over all $\nu\in \bbZ^d$ with 
$8\le \nu_i\le 2^{n_l-3}$ for $i=1,\dots, d$. Notice that 
$J_{l,\eps}$ is contained in a cube of sidelength $1/2$ centered at $le_1$. 
By the condition \eqref{lowerbdonpsipsi} we have 
$|\cL_{n_l} g_{n_l,\nu}(x)| \ge c $ for $x\in J_{l,\nu,\eps}$.
We have
$$\Big(\sum_{l=1}^N \big |a_l \cL_{n_l} g_{n_l}(x)\big |^q\Big)^{1/q} \ge G_{\text{low}}(x)$$
where
$$G_{\text{low}}(x) = \sum_{l=1}^N |a_l| \sum_{\nu=2^3}^{2^{n_l-3}} \bbone_{J_{l,\nu, \eps}}(x).$$
Note that  the measure of $J_{l,\eps} $ is at least $c_0 \eps^d$ for some fixed positive $c_0$. 
Hence 
$$G_{\text{low}}^*(t) \ge \sum_{l=1}^N |a_l|\bbone_{(c_0\eps^d(l-1), c_0 \eps^d l ]}(t) $$ and thus
\[I'\ge \|G_{\text{low}}\|_{L^{p,r}}\ge c'\Big(\sum_{l=1}^N l^{\frac rp-1} |a_l|^r\Big)^{1/r}.\]

For $II'$  we get a better upper bound. By the argument for the upper bound above we 
obtain due to  separateness condition of the $n_l$
\[ II' \le C2^{-R} \Big(\sum_{l=1}^N l^{\frac rp-1} |a_l|^r\Big)^{1/r}.\] Thus if  $R$  in our  definition of the $n_l$ is chosen large enough
we get the lower bound
\[\|f\|_{F^s_q[L^{p,r}]} \ge c'' (\sum_{l=1}^N l^{\frac rp-1} |a_l|^r)^{1/r}\]
provided that the right hand side is finite.
\end{proof}




\subsubsection{Conditions on $q_0$, $q_1$ when $p_0=p_1$, $s_0=s_1$.}\label{q0q1necsect}
Let $\chi\in C^\infty_c(\bbR^d)$ be supported in $\{\xi:|\xi|\le 2^{-4}\}$ such that $\chi (0)=1$, and $\widehat \chi(0)=1$. Let, given  $N\in \bbN$,
$f(x)=\sum_{l=N+1}^{2N} 2^{-ls} a_l \eta_l(x)$ with
  $\eta_l= \cF^{-1}[\chi(\cdot-2^l e_1)]$.
Using the support properties of $\beta_k$ (\cf. \S\ref{scalingexample})
we have, for $l> 1$,  $\La_l\eta_l=\eta_l$ and $\La_k\eta_l=0$ for $k\neq l$. Hence for large $N$
\begin{align*}
\|f\|_{B^s_{q_0}[L^{p,r}]}&\approx 
\Big(\sum_{l=N+1}^{2N} |a_l|^{q_0}\Big)^{1/q_0},
\\
\|f\|_{F^s_{q_1}[L^{p,r}]}&\approx 
\Big(\sum_{l=N+1}^{2N} |a_l|^{q_1}\Big)^{1/q_1}.
\end{align*} 
This immediately yields that  every  of the embeddings
$B^s_{q_0}[L^{p,r_0}] \hookrightarrow F^s_{q_1}[L^{p_1,r_1}]$,
$F^s_{q_0}[L^{p,r_0}] \hookrightarrow B^s_{q_1}[L^{p_1,r_1}]$,
$B^s_{q_0}[L^{p,r_0}] \hookrightarrow B^s_{q_1}[L^{p_1,r_1}]$,
$B^s_{q_0}[L^{p,r_0}] \hookrightarrow B^s_{q_1}[L^{p_1,r_1}]$ implies that  $q_0\le q_1$.


\subsubsection{Necessary conditions on $(p,r_0) $ and $(p, r_1)$ in Theorems \ref{BinFthm}, \ref{FinBthm}, (vi)}
\label{spequality-nec-pr}
We now show 
\begin{alignat}{2}\label{cexbf4}
& B^s_p[L^{p,r_0}] \hookrightarrow F^s_p[L^{p,r_1}] &&\implies r_0\le p.
\\ \label{cexbf8}
&F^{s}_{p}[L^{p,r_0}] \hookrightarrow  B^s_p[L^{p,r_1}] &&\implies p\le r_1.
\end{alignat}
For both implications we choose large parameters $R, N\in \bbN$ and  use as a test function
\begin{subequations}\label{fwithalog} 
\Be 
f(x)=\sum_{k=N+1}^{2N} 2^{-sR k} \sum_{l=0}^{4N} (1+|R k-R l|)^{-\frac 1p} (\log(2+|R k-R l |))^{-\delta} f_{k,l}(x),\Ee
where \Be f_{k,l}(x)= \cF^{-1}[\chi(\cdot-2^{R k} e_1)](x-R le_1),
\Ee
with $\chi$ as in \S\ref{q0q1necsect}.
\end{subequations}

The parameter $\delta$ will be suitably 
chosen depending on the parameters $r_0,p$ or $r_1,p$ in \eqref{cexbf4}, \eqref{cexbf8}, respectively.
Let $\eps>0$ such that
$$|\cF^{-1}[\chi] (x)|\ge c>0 \text{  for }\, x\in [-\eps,\eps]^d.$$ 
We  have  $\La_{R k}  f_{k,l}=  f_{k,l} $,  and 
$\La_{j}  f_{k,l}= 0$ when $j\neq R k$.
Let $\eta= \cF^{-1}[\chi]$ then $f_{k,l} (x)=\eta(x-R l e_1) e^{2\pi i  2^{R k} x_1}$. Hence
\begin{multline*} 
2^{s R k}| \Lambda_{R k} f(x)| = 
\Big|\sum_{l\in \bbZ} (1+|R k-R l|)^{-\frac 1p} (\log(2+R|k-l|))^{-\delta} \eta(x-R l e_1)\Big| \\\text{ if } N+1\le k\le 2N
\end{multline*} and 
$2^{sj}| \Lambda_{j} f(x)| =  0 $ if $j\notin R \bbN$ or if $j\notin  [R(N+1),2RN].$

\begin{proof}[Proof of \eqref{cexbf4}]
We argue by contradiction and assume  $r_0>p$. Choose $\delta$ so that $1/r_0<\delta<1/p$
and let  $f$ be as in
\eqref{fwithalog}.  We then have
for fixed $k$ that
$$\Big\|  \sum_{l\in \bbZ} (1+|R k-R l|)^{-\frac 1p} (\log(2+R|k-l|))^{-\delta}
 \eta(\cdot-R l e_1)\Big\|_{p,r_0}\lc 1$$ and hence 
 \Be\label{fBupp}\|f\|_{B^s_p(L^{p,r_0})} \lc N^{1/p}.\Ee
 
To derive a lower bound for  $\|f\|_{F^s_p [L^{p,\infty}]} $ we let
$$\cU_\eps=\{x: |x-k_0Re_1 |\le \eps \text{ for some $k_0\in [5N/4, 7N/4]$}\}.$$
Fix $x\in \cU_\eps$.  Then
$(\sum_j|2^{js} \Lambda_j f(x)|^p)^{1/p}\ge c_1 I(x)-c_2II(x)$ where
\[I(x)= c\Big(\sum_{k=N+1}^{2N}(1+|Rke_1-x|)^{-1} (\log(2+|Rke_1-x|))^{-\delta p}\Big)^{1/p}\]
and $II(x)=$ \begin{multline*}
C\Big(\sum_{k=N+1}^{2N}\Big|\sum_{l\neq k} (1+R|k-l|)^{-1/p} (\log(2+R|k-l|))^{-\delta}
\eta(x-{R l}e_1)\Big|^p\Big)^{1/p}.
\end{multline*}
Now 
$I(x)\gc (R^{-1} \sum _{2\le l\le N/4} l^{-1}  (\log l)^{-\delta p} )^{1/p}
\gc R^{-1/p}
(\log N)^{\frac{1-\delta p}{p}}$, as $\delta p<1$. Using the decay of $\eta$ we also get
$II(x)\lc C_{N_1} R^{-N_1} (\log N)^{\frac{1-\delta p}{p}}$ for any $N_1>0$.  Hence for $R$ large we see that the measure of the subset of $\cU_\eps$  where $(\sum_j|2^{js} \Lambda_j f(x)|^p)^{1/p}\ge  
c(\log N)^{\frac{1-\delta p}{p}}$  is bounded below by  times $c\eps^d N$. Hence
\Be \label{fFlow}\|f\|_{F^s_p(L^{p,r_1})}
\gc \|f\|_{F^s_p(L^{p,\infty})}
\gc_\eps  N^{1/p} (\log N)^{\frac{1-\delta p}{p}}.\Ee
Comparing \eqref{fBupp} and \eqref{fFlow}, and choosing $N$ large, we get a  contradiction. Hence $r_0\le p$.
\end{proof}
\begin{proof}[Proof of  \eqref{cexbf8}]
Again we argue by contradiction and assume  $r_1<p$. Let $\delta$ be such that
 $1/p<\delta<1/r_1$ and let  $f$ be as in
\eqref{fwithalog}. 
Since $\delta>1/p$ we have, for any $M_1$,
$$\Big(\sum_{ k} 2^{R k s p} | \La_{R k} f(x)|^p\Big)^{1/p} \le C_{M_1}\begin{cases} 1 &\text{ if } -N\le x\le 2N
\\
(1+|x|)^{-M_1} &\text{ otherwise.}
\end{cases}
$$
This gives  
\Be
\label{Lpr0ellpupp}
\Big\|\Big(\sum_{k} 2^{R k s p} | \La_{R k} f|^p\Big)^{1/p} \Big\|_{p,r_0} \lc N^{1/p}, \quad r_0>0.
\Ee
On the other hand we claim that 
\Be \label{lowerbdk} 2^{R k s} \|  \La_{R k} f\|_{p,r_1} \gc (\log N)^{-\delta+1/r_1},\quad
5N/4\le k\le 7N/4.
\Ee
Let $\cV_{\eps,k,j}= \cup_{2^{j-1}\le k-l\le 2^j} B(Rle_1, \eps)$, for $j$ with $0<2^j\le N/4$.
We have
\begin{align*}2^{Rks}|\La_{Rk} f(x)|&= \Big|\sum_{l=0}^{4N}
(1+R|k-l|)^{-1/p} \log (2+R|k-l|)^{-\delta} \eta(x-Rle_1)\Big|
\\& \ge c_1 I_k(x)-C_1 II_k(x), \quad\text{ with}
\end{align*}
\begin{align*}
I_k(x)&= \sum_{0<2^j< N/4}(1+R 2^j)^{-1/p} \log (2+R2^j)^{-\delta} \bbone_{\cV_{\eps,k,j}}(x),
\\
II_k(x) &= \sum_{\substack{0\le l\le 4N\\ |Rle_1-x|\ge R/2}} 
\frac{\log(2+R|k-l|)^{-\delta}}{
(1+R|k-l|)^{1/p} }
 R^{-N_1}|x-Rl e_1 |^{-N_1}.
\end{align*}
It is immediate that $\|II_k\|_p \lc  R^{-N_1}$, for any $N_1$, 
 and by interpolation also 
$\|II_k\|_{p,r_1} \lc  R^{-N_1}$.

Notice that  $\meas(\cV_{\eps, k, j})\ge c\eps^d 2^j$.
For the rearrangement of $I_k$ we have
$$I_k^*(t) \ge c \sum_{0\le 2^j\le N/8} 2^{-j/p}R^{-1/p} [\log(R2^j)]^{-\delta}(t) 
\bbone_{[0,c\eps^d 2^j]}(t)$$ and thus
\begin{align*}
\|I_k\|_{p,r_1} &\ge \Big(\sum_{0<2^j<N/8} \int_{c\eps 2^{j-1}}^{c\eps 2^j} \big[
(R2^{j})^{-1/p} 
(\log(R2^j))^{-\delta} \big]^{r_1} t^{r_1/p}\frac{dt}{t}\Big)^{1/r_1}
\\ &\ge R^{-1/p} (\log N)^{\frac{1-\delta r_1}{r_1}}.
\end{align*}
The two estimates for $\|I_k\|_{p,r_1} $ and $\| II_k \|_{p,r_1}$ imply \eqref{lowerbdk}. Then also
\Be\label{fBlowbd} 
\|f\|_{B^s_p(L^{p,r_1})} \gc R^{-1/p} N^{1/p}(\log N)^{-\delta+1/r_1}.\Ee 
Comparing \eqref{Lpr0ellpupp} and \eqref{fBlowbd}, and choosing  $N$ large,  we  get a contradiction when $\delta<1/r_1$. This means we must have $r_1\ge p$ in \eqref{cexbf8}.
\end{proof}



\section{Sequences of vector-valued functions}\label{vectvalemb}
In order to prove the  positive results in parts (iv)-(vi) of Theorems \ref{BinFthm} and  \ref{FinBthm} we derive  corresponding embeddings  for spaces of sequences 
$\ell^{q}(L^{p,r})$ and $L^{p,r}(\ell^q)$, for fixed $p,r$. 

\begin{proposition}
 \label{btlprepthm}
Let $0<p<\infty$, $0< q_0,q_1, r_0, r_1\le \infty$ 
and assume $q_0\le\min\{p,q_1,r_1\}$, $r_0\le r_1$.
The embedding 
$$\ell^{q_0}(L^{p,r_0})\hookrightarrow L^{p,r_1}(\ell^{q_1})$$
holds in each of the following three cases:

(i) $p\neq q_1$,  (ii) $p=q_1\ge r_0$,  (iii) $q_0<p=q_1<r_0$.

\end{proposition}

\begin{proposition} \label{btldualprepthm}
Let $0<p<\infty$, $0< q_0,q_1, r_0, r_1\le \infty$, $r_0\le r_1$ and 
$q_1\ge \max\{p, q_0, r_0\}$.
The embedding 
$$
L^{p,r_0}(\ell^{q_0})\hookrightarrow 
\ell^{q_1}(L^{p,r_1}) 
$$
holds in each of the following three cases:

(i) $p\neq q_0$,  (ii) $p=q_0\le r_1$, (iii) 
$r_1<q_0=p<q_1$.
 \end{proposition}




\begin{remark} In both Proposition \ref{btlprepthm} and Proposition \ref{btldualprepthm} the assumptions (i), (ii), (iii) cannot be improved (unless one imposes very restrictive conditions  on the  underlying measure spaces). 
This follows from the examples for the spaces $B^s_q[L^{p,r}]$, $F^s_q[L^{p,r}]$ discussed in 
 \S\ref{meassmoothness}, although one can give less technical examples for the propositions. \end{remark}



We split the proof into several lemmata.
\begin{lemma}\label{vectineq-prqlem}
Suppose  $  q\le r\le p \text{ or  }q<p\le r$. Then
\Be
\label{p>q}
\ell^q(L^{p,r})\hookrightarrow L^{p,r}(\ell^q).
\Ee
\end{lemma}

\begin{proof}
The asserted inequality is trivial when $p=q=r$.
We may thus assume $p>q$.  Then

\begin{align*}
\| f\|_{L^{p,r}(\ell^q)}^q
=\Big\| \sum_k |f_k|^q \Big\|_{L^{p/q, r/q}}
\lc 
\sum_k
\big\|  |f_k|^q \big\|_{L^{p/q, r/q}}
=
 \sum_k \|f_k\|_{L^{p,r}}^q .
\end{align*}
Here we have used the triangle inequality in  
\eqref {trianglep>1}, for the space $L^{p/q, r/q}$, and twice the formula 
\eqref{lor-power}.
\end{proof}

\begin{lemma} \label{dualitylem}
 Suppose that  either $p\le r\le q$  or  $ r\le p<q$. Then
\[
L^{p,r}(\ell^q)\hookrightarrow \ell^q(L^{p,r}).
\]
\end{lemma} 

\begin{proof}
We first consider the case  $r<\infty$ and argue by duality.
Recall that if $A$ is a Banach space, $A'$  its dual and $1<u<\infty$ then the dual of $\ell^u(A)$ is $\ell^{u'}(A')$,
with the natural pairing.
Let $a<\min \{p,q,r\}$ and set $(P,Q,R)= (p/a,q/a, r/a)$ so that $1<P,Q,R<\infty$.
Since for $1<P,R<\infty$ the dual of $L^{P,R}$ is $L^{P',R'}$ we see that
\begin{align*}&\| f\|_{\ell^q(L^{p,r})}^a=
\Big( \sum_k\|f_k\|_{L^{p,r}}^q\Big)^{a/q}
\\&= 
\Big(\sum_k\| |f_k|^a \|_{L^{P,Q}}^Q\Big)^{1/Q}
\le 
\Big(\sum_k\trn |f_k|^a \trn_{L^{P,Q}}^Q\Big)^{1/Q}
\\&\lc \sup \Big\{ \sum_k\int|f_k(x)|^a g_k(x) d\mu(x) \,: \trn g\trn_{\ell^{Q'}(L^{P',R'})}\le 1\Big\}
\end{align*}
where the implicit constants depend on $p,q,r$.
Now, let $\trn g\trn_{\ell^{Q'}(L^{P',R'})} \le 1$. Then
\begin{align*}
\sum_k \int   |f_k(x)|^a&|g_k(x)| d\mu \le 
\int\Big(\sum_k|f_k(x)|^{aQ}\Big)^{1/Q} \Big(\sum_k|g_k(x)|^{Q'}\Big)^{1/Q'} d\mu
\\
&\lc \Btrn
\Big(\sum_k|f_k|^{aQ}\Big)^{1/Q} \Btrn_{L^{P,R}} \,\Btrn \Big(\sum_k|g_k|^{Q'}\Big)^{1/Q'} \Btrn_{L^{P',R'}}
\\
&\lc \Big\| \Big(\sum_k|f_k|^{aQ} \Big)^{1/aQ} \Big\|_{L^{aP,aR}}^a
\trn g\trn_{\ell^{Q'}(L^{P',R'})} 
\,\lc \big\|f\big\|_{L^{p,r}(\ell^q)}^a ;
\end{align*}
here we have used for the second to last inequality that 
$$
 \trn g \trn_{L^{P',R'}(\ell^{Q'})} \lc \| g \|_{L^{P',R'}(\ell^{Q'})} \lc\|g\|_{\ell^{Q'}(L^{P',R'})},
$$ by Lemma \ref{vectineq-prqlem} since $Q'\le R'\le P'$ or $Q'<P'\le R'$.
This completes the proof for  $r<\infty$.

Next assume $r=\infty$, then also $q=\infty$. 
Clearly we have 
for any fixed $k_0$ 
\[\mu(\{x: |f_{k_0}(x) |>\alpha\}) \le 
\mu(\{x: \sup_k|f_{k}(x) |>\alpha\}).
\]
Hence
$ \sup_k \sup_{\alpha>0}  \alpha[\mu_{f_k}(\alpha)]^{1/p}
\le \sup_{\alpha>0}  \alpha [ \mu_{\sup_k|f_k|}(\alpha)]^{1/p}$
which yields the case  $r=q=\infty$.
\end{proof}

Next we state some weaker embedding properties for the case $p<q\le r$.

\begin{lemma}\label{ppr-prr}
(i) Let $p<q\le r$ or $p=q=r$. Then 
\[\ell^p(L^{p,r})\hookrightarrow L^{p,r}(\ell^q)\,.\]
 (ii) Let $v<p\le r$. Then 
\[\ell^v(L^{p,r})\hookrightarrow L^{p,r}(\ell^p).\]
\end{lemma}

\begin{proof} The statement is trivial for $q=r=p$.
Let $p<q\le r$.  We use the modified $p/q$-triangle inequality in $L^{p/q,s}$ for $s=r/q\ge 1$, as in  \eqref{ptrianglelem}, and estimate
\begin{align*}
&\Big\|\Big(\sum_k|f_k|^q\Big)^{1/q}\Big\|_{L^{p,r}} = \Big\|\sum_k |f_k|^q \Big\|_{L^{p/q,r/q}}^{1/q}
\\
&\lc
\Big( \Big(   \sum_k  \big \| |f_k|^q \big\|_{L^{p/q,r/q}}^{p/q}\Big )^{{q}/{p}}\Big)^{ 1/q}
=
\Big(\sum_k  \big \|f_k \big\|_{L^{p,r}}^{p}\Big )^{1/p}\,. 
\end{align*}

For (ii) we  use the embedding $\ell^v\subset \ell^p$ and then the triangle inequality in $L^{p/v, r/v}$ (cf.  \eqref{trianglep>1}) to obtain
\begin{align*}
&\Big\|\Big(\sum_k|f_k|^p\Big)^{1/p}\Big\|_{L^{p,r}} \le  \Big\|\Big(\sum_k |f_k|^v \Big)^{1/v}\Big\|_{L^{p,r}}
\\
 &\le  \Big\|\sum_k |f_k|^v \Big\|_{L^{p/v,r/v}}^{1/v} \lc \Big( 
 \sum_k \| |f_k|^v \|_{L^{p/v,r/v}}\Big)^{1/v}  \lc 
  \Big(\sum_k \|f_k\|_{L^{p,r}}^v\Big)^{1/v}.  \qedhere
\end{align*}
\end{proof}

\begin{lemma}\label{secdualitylem}
(i) Let $0< r\le q< p$ or $r=p=q$. Then
\[
L^{p,r}(\ell^q)\hookrightarrow \ell^p(L^{p,r}).
\]

(ii) Let $ 0< r\le p<w$. Then
\[
L^{p,r}(\ell^p)\hookrightarrow \ell^w(L^{p,r}).
\]
\end{lemma}
\begin{proof}
The statement is trivial for $r=p=q$. If $0<r\le q<p<\infty$,  set
 $(P,Q,R)= (p/a,q/a, r/a)$ for some $a<\min\{ p,q,r\}$ and argue by duality exactly as in the proof of Lemma 
 \ref{dualitylem}, basing the argument on Lemma \ref{ppr-prr}.
 \end{proof}

\begin{proof} [Proof of Proposition \ref{btlprepthm}]
Let $q_0\le \min \{p,q_1,r_1\}$. We  distinguish the three cases according to whether $p$, $q_1$, or $r_1$ is the smallest exponent.

{\it Case 1:} $q_0\le q_1\le \min\{p,r_1\}$.  
If either $q_0\le r_1\le p$ or $q_0<p\le r_1$
then $\ell^{q_0}(L^{p,r_1})\hookrightarrow L^{p,r_1}(\ell^{q_0})$ by Lemma \ref{vectineq-prqlem} 
and hence
\[\ell^{q_0}(L^{p,r_0}) \hookrightarrow
\ell^{q_0}(L^{p,r_1})\hookrightarrow L^{p,r_1}(\ell^{q_0})
\hookrightarrow L^{p,r_1}(\ell^{q_1}).
\] 
In the remaining subcase we have $q_0=q_1=p\le r_1$ and  by assumption of the proposition we also have 
$r_0\le p$. Thus 
$\ell^{p}(L^{p,r_0})\hookrightarrow \ell^p(L^p)=L^p(\ell^p)\hookrightarrow L^{p,r_1}(\ell^p)$.

{\it Case 2:}  $q_0\le r_1\le\min\{p,q_1\}$. Note that 
$\ell^{r_1}(L^{p,r_1})\hookrightarrow L^{p,r_1}(\ell^{r_1}) $ for $r_1\le p$,  again by  Lemma  \ref{vectineq-prqlem}. Thus we obtain
\[\ell^{q_0}(L^{p,r_0}) \hookrightarrow \ell^{r_1}(L^{p,r_1})\hookrightarrow L^{p,r_1}(\ell^{r_1}) \hookrightarrow L^{p,r_1}(\ell^{q_1}).
\]

In the third case $q_0\le p\le \min\{q_1,r_1\}$. The embedding is trivial when $p=q_1=r_1$.
We distinguish three remaining subcases.

{\it Case 3-1:}  $p<q_1\le r_1$. We apply Lemma \ref{ppr-prr}, (i), 
to get 
\[\ell^{q_0}(L^{p,r_0})\hookrightarrow 
\ell^p(L^{p,r_1})\hookrightarrow  
L^{p,r_1}(\ell^{q_1}).\]

{\it Case 3-2:} $q_0\le p=q_1<r_1$. 
If   $q_0<p$  then by part (ii) of  Lemma \ref{ppr-prr}.
\[\ell^{q_0}(L^{p,r_0})  \hookrightarrow \ell^{q_0}(L^{p,r_1})\hookrightarrow L^{p,r_1}(\ell^p).\]
If $q_0=p$ then we have by assumption also $r_0\le p$ and therefore 
$\ell^p(L^{p,r_0})\hookrightarrow \ell^p(L^p)= L^p(\ell^p)\hookrightarrow L^{p,r_1}(\ell^p)$.

{\it Case 3-3:}   $q_0\le p\le r_1<q_1$. Now  observe that 
$\ell^p(L^{p,r_1})\hookrightarrow L^{p,r_1}(\ell^{r_1})$ 
by  Lemma  \ref{ppr-prr}, (i). Hence
\[ \ell^{q_0}(L^{p,r_0})\hookrightarrow 
\ell^p(L^{p,r_1})\hookrightarrow  
L^{p,r_1}(\ell^{r_1})\hookrightarrow  L^{p,r_1}(\ell^{q_1})\,.  \qedhere\]
\end{proof}

\begin{proof} [Proof of Proposition \ref{btldualprepthm}]
The proof is `dual' to the proof of 
Proposition  \ref{btlprepthm}. Formally the proof goes by reversing the arrows in the proof of Proposition \ref{btlprepthm} and replacing the subscripts $(0,1)$ by $(1,0)$. We now use Lemma \ref{dualitylem} and 
Lemma \ref{secdualitylem} in place of Lemma \ref{vectineq-prqlem} and
Lemma \ref{ppr-prr}. We  run through the cases:

{\it Case 1':} $q_1\ge q_0\ge \max\{p,r_0\}$. If 
$p\le r_0\le q_1$, or $r_0\le p< q_1$ and  the second of the following embeddings holds by 
Lemma \ref{dualitylem}:
\[ L^{p,r_0}(\ell^{q_0}) \hookrightarrow
L^{p,r_0}(\ell^{q_1}) \hookrightarrow
\ell^{q_1}(L^{p,r_0}) \hookrightarrow
\ell^{q_1}(L^{p,r_1}).
\]
If  $r_0\le p=q_0=q_1$ then also by assumption $r_1\ge p$ and hence $L^{p,r_0}(\ell^{q_0})\hookrightarrow L^p(\ell^p)=\ell^p(L^p)\hookrightarrow \ell^p(L^{p,r_1})$.

{\it Case 2':} $q_1\ge r_0\ge \max \{p,q_0\}$. First observe that
Lemma \ref{dualitylem} also implies 
$L^{p,r_0}(\ell^{r_0})\hookrightarrow \ell^{r_0}(L^{p,r_0})$ for $p\ge r_0$.
Hence
\[ L^{p,r_0}(\ell^{q_0}) \hookrightarrow
L^{p,r_0}(\ell^{r_0}) \hookrightarrow
\ell^{r_0}(L^{p,r_0}) \hookrightarrow
\ell^{q_1}(L^{p,r_1}).
\]

The third case $q_1\ge p\ge \max \{q_0, r_0\}$ is again split into three subcases (ignoring the trivial case $p=q_0=r_0$).

{\it Case 3-1':} $p>q_0\ge r_0$. We apply Lemma \ref{secdualitylem} to obtain
\[ L^{p,r_0}(\ell^{q_0}) \hookrightarrow
\ell^p(L^{p,r_0}) \hookrightarrow
\ell^{q_1}(L^{p,r_1}).
\]

{\it Case 3-2':} $q_1\ge p=q_0>r_1$. If 
$q_1>p$ we get by part   (ii) of Lemma \ref{secdualitylem},
\[ L^{p,r_0}(\ell^p) \hookrightarrow
\ell^{q_1}(L^{p,r_0}) \hookrightarrow
\ell^{q_1} (L^{p,r_1}). 
\]
If $q_1=p$ then by assumption $r_1\ge p$ and therefore 
$L^p(\ell^{r_0})\hookrightarrow L^p(\ell^p)=\ell^p(L^p)\hookrightarrow \ell^p(L^{p,r_1})$.

{\it Case 3-3':} $q_1 \ge p\ge r_0>q_0$. We use  that $L^{p,r_0}(\ell^{r_0})\hookrightarrow \ell^p(L^{p,r_0})$, by Lemma \ref{secdualitylem}.
Hence 
\[L^{p,r_0}(\ell^{q_0}) \hookrightarrow
L^{p,r_0}(\ell^{r_0}) \hookrightarrow
\ell^p(L^{p,r_0}) \hookrightarrow
\ell^{q_1}(L^{p,r_1}). \qedhere
\]
\end{proof}

We now get the statements (iv)-(vi) in Theorems \ref{BinFthm} and \ref{FinBthm}.

\begin{corollary}
 \label{btlthm}
Let $0<p<\infty$, $0< q_0,q_1,r_0, r_1\le \infty$. 

(i)  Suppose that either $p\neq q_1$ or that $p=q_1\ge r_0$. 
Then the embedding 
\Be \label{BvinFq} 
B^s_{q_0}[L^{p,r_0}] \hookrightarrow F^s_{q_1}[L^{p,r_1}]
\notag\Ee  holds if and only if $q_0\le\min\{p,q_1,r_1\}$ and $r_0\le r_1$.

(ii) Let $r_0>p$. Then
the embedding 
 $$B^s_{q_0}[L^{p,r_0}] \hookrightarrow F^s_{p}[L^{p,r_1}]$$  holds if and only if $q_0<p$ and $r_0\le r_1$.
\end{corollary}


\begin{corollary} \label{btldualthm}
Let $0<p<\infty$, $0< q_0,q_1, r_0, r_1\le \infty$. 

(i) Suppose  either  that $p\neq q_0$ or that $p=q_0\le r_1$. 
Then the embedding 
\Be\label{FqinBw}
F^s_{q_0}[L^{p,r_0}]\hookrightarrow 
B^s_{q_1}[L^{p,r_1}]
\notag
\Ee
holds  if  and only if $q_1\ge\max\{p,q_0,r_0\}$ and $r_0\le r_1$. 

(ii) Let $r_1<p$. Then 
the  embedding $$F^s_p[L^{p,r_0}]\hookrightarrow B^s_{q_1}[L^{p,r_1}] $$ 
holds  if and only if  $q_1>p$ and $r_0\le r_1$. 
\end{corollary}

\begin{proof} [Proof of  Corollary \ref{btlthm} and Corollary \ref{btldualthm}]
The positive   results  follow immediately from the corresponding results in
Propositions \ref{btlprepthm} and \ref{btldualprepthm}  
when  applied  to
$\{ f_k\}_{k=0}^\infty$   with $f_k=2^{ks}\La_kf$. The necessity of the conditions was proved in \S\ref{meassmoothness}.
\end{proof}

\section{Embeddings of Jawerth-Franke type}
\label{jawerthfrankesect}
Jawerth's and Franke's  versions of the Sobolev embedding theorem  were  reproved by Vyb\'iral \cite{vyb} using sequence spaces  which are discrete variants of  Besov and Triebel-Lizorkin spaces.
The proofs are inspired by  \cite{vyb}.  For reasons of brevity and preference  we  choose not to introduce   sequence spaces in the Lorentz category.

\subsection*{\it Preliminary considerations}
We first need a straightforward  Lorentz space version of Peetre's maximal theorem.
\begin{lemma}
Let $f_k\in \cS'$ be such that $\widehat f_k$ is supported in $\{\xi:|\xi|\le 2^k\}$.
Let
$$\fM_k f_k(x) =\sup_{|h|\le d 2^{-k}} |f_k(x+h)|\,.$$
Then for $0<p<\infty$, $0<q,r\le \infty$,
\[
\big\|\{\fM_k f_k\}\big\|_{L^{p,r}(\ell^q)} \lc C_{p,q,r} 
\big\|\{ f_k\}\big\|_{L^{p,r}(\ell^q)} .
\]
\end{lemma} 

\begin{proof} 
Let $M\ci{HL}$ be the Hardy-Littlewood maximal operator.
We have 
\Be\label{HL}\|\{M\ci{HL} g_k\}\|_{L^{p_0,r_0}(\ell^{q_0})} 
\le C(p_0,r_0, q_0) 
\|\{ g_k\}\|_{L^{p_0,r_0}(\ell^{q_0})} 
\Ee
for $1<p_0, r_0, q_0<\infty$. 
The version for $p_0=r_0$ was  proved by Fefferman and Stein \cite{fs} and the general version follows by real interpolation.

From  \cite{Pe}  we have the inequality 
$$\fM_k f_k(x) \le C_{\rho} \big(  M\ci{HL} (|f_k|^\rho)(x) \big)^{1/\rho}$$
for all $\rho>0$. We choose $\rho<\min \{p,q,r\}$, and apply \eqref{HL} with
$(p_0, r_0, q_0)= (p/\rho, r/\rho, q/\rho)$.
Then 
\begin{align*}
&\big\| \{ \fM_kf_k\}\big\|_{L^{p,r}(\ell^q)}
\lc 
\big\| \{ ( M\ci{HL} (|f_k|^\rho))^{1/\rho}\}\big\|_{L^{p,r}(\ell^q)}
\\&=
\big\| \{ M\ci{HL} (|f_k|^\rho)\}\big\|_{L^{\frac{p}\rho, \frac r\rho}(\ell^{\frac q\rho})}^{1/\rho} 
\lc 
\big\| \{ |f_k|^\rho\}\big\|_{L^{\frac p\rho, \frac r\rho}(\ell^{\frac q\rho})}^{1/\rho} 
=
\big\| \{ f_k\}\big\|_{L^{p,r}(\ell^q)}.\qedhere
\end{align*}
\end{proof}

\begin{theorem} \label{jawerththm} 
Suppose $0<d(1/p_0-1/p_1)=s_0-s_1$,  $0<q_0,q_1,r_0,r_1\le \infty$.
Then the embedding
$
F^{s_0}_{q_0} [L^{p_0,r_0}] \hookrightarrow
B^{s_1}_{q_1} [L^{p_1,r_1}]
$
holds if and only if $r_0\le q_1$.
\end{theorem}

\begin{proof} The necessity of the condition $r_0\le q_1$ has been established in \S\ref{franke-jawerth-nec}.

 Let $Q_k(x)$ be the unique dyadic cube of sidelength $2^{-k}$ which contains $x$, (the sides being half open intervals). Set
 \begin{align*}
 g_k(x)&= 2^{ks_0}  \sup_{y\in Q_k(x)}
  \La_kf(y),
 \\
 G(x)&=\sup_k \,2^{ks_0}  \fM_k \La_k f(x).
\end{align*}
Clearly \[g_k(x) \le 2^{ks_0}\, \fM_k \La_kf (x) \le G(x)\]
and therefore $g_k^*(t)\le G^*(t)$.

Since $g_k$ is constant on the dyadic cubes of sidelength $2^{-k}$ we see that $g_k^*$ is constant on dyadic intervals of length $2^{-kd}$.
In particular
\begin{align*}
\|g_k\|_{p_1,r_1} &= \Big( \frac{r_1}{p_1} \int_0^\infty t^{r_1/p} g_k^*(t)^{r_1} \frac{dt}{t}\Big)^{1/r_1}
\\&\approx \Big( \sum_{n=1}^\infty  
\big[(2^{-kd} n)^{1/p_1} g_k^*(2^{-kd}(n-1)) 
\big]^{r_1} \frac{ 2^{-kd}}{2^{-kd}n}\Big)^{1/r_1}\,.
\end{align*}

We now begin with the proof of the sufficiency of the condition $q_1\ge r_0$.  Since the 
$ B^{s}_{q} (L^{p,r})$, $ F^{s}_{q} (L^{p,r})$ norms increase when $r$ decreases or when $q$ decreases, it suffices to consider the case $q_1=r_0=:\rho$ and to prove 
for $0<\rho\le \infty$  and $r_1<\rho$ 
\Be\label{jawerthextreme}
F^{s_0}_{\infty}[L^{p_0,\rho}] \hookrightarrow B^{s_1}_{\rho}[L^{p_1,r_1}].
\Ee
 We write the proof for $\rho<\infty$ but this is not essential as the case $\rho=\infty$ will only require  notational changes. 
 
 Now
 \begin{align*}
 \|f\|_{ B^{s_1}_{\rho} (L^{p_1,r_1})} &=\Big(\sum_{k=0}^\infty 2^{ks_1 \rho} \|\La_k f\|_{p_1,r_1}^{\rho} \Big)^{1/\rho}
 \\&= 
 \Big(\sum_{k=0}^\infty  \| 2^{ks_0} \La_k f\big \|_{p_1,r_1}^{\rho}2^{-kd(\frac{1}{p_0}-\frac {1}{p_1})\rho }\Big)^{1/\rho}
 \end{align*}
 Here we have used the relation $d/p_0 - d/p_1\,=s_0-s_1$.
 The last displayed expression is dominated by
\begin{align*}
&\Big(\sum_{k=0}^\infty  2^{-kd(\frac 1{p_0}-\frac 1{p_1})\rho} \| g_k\|_{p_1,r_1}^{\rho} 
\Big)^{1/\rho}
 \\
 &\lc \Big( \sum_{k=0}^\infty
 2^{-kd(\frac{1}{p_0}-\frac {1}{p_1})\rho}\Big[ \sum_{n=1}^\infty (2^{-kd}n)^{r_1/p_1} 
 g_k^*(2^{-kd}(n-1))^{r_1} \frac1{n}\Big]^{\rho /r_1} \Big)^{1/\rho}
 \\
 &\le \Big(\sum_{k=0}^\infty \Big(\sum_{n=1}^\infty n^{r_1(\frac 1{p_1}-\frac{1}{p_0})-1 } G^*(2^{-kd}(n-1))^{r_1}
  (2^{-kd}n)^{r_1/p_0}
 \Big)^{\rho/r_1}\Big)^{1/\rho}\,.
 \end{align*}
The last  expression is comparable to
\begin{align*}
 &\Big(\sum_{k=0}^\infty \Big(\sum_{j=0}^\infty 2^{jr_1( 1/{p_1}-{1}/{p_0}) } G^*(2^{-kd+j-1})^{r_1}
  (2^{-kd+j})^{r_1/p_0}
 \Big)^{\rho/r_1}\Big)^{1/\rho}
 \\
  & \lc
\Big(\sum_{j=0}^\infty 2^{jr_1({1}/{p_1}-{1}/{p_0})} 
 \Big( \sum_{k=0}^\infty G^*(2^{-kd+j-1})^{\rho} (2^{-kd+j})^{\rho/p_0} \Big)^{r_1/\rho} \Big)^{1/r_1}
 \\  & \lc
\Big(\sum_{j=0}^\infty 2^{jr_1(1/p_1-1/p_0)} 
 \Big( \sum_{k=0}^\infty \int_{2^{-kd+j-1}}^{2^{-kd+j}} t^{\rho/p_0} G^*(t)^{\rho} \frac{dt}{t} \Big)^{r_1/\rho }\Big)^{1/r_1}.
 \end{align*}
 Here we have used the triangle inequality in $L^{\rho/r_1}$  (as  $\rho/r_1\ge 1$).
Now for fixed $j$ the intervals $[2^{-kd+j-1}, 2^{-kd+j}]$ have disjoint interior and therefore the last expression is dominated by
\begin{align*}
&\Big(\sum_{j=0}^\infty 2^{jr_1(1/p_1-1/p_0)} \|G\|_{p_0,\rho}^{r_1}\Big)^{1/r_1}
\\ &\lc \|G\|_{p_0,\rho} \lc \big\|\sup_k 2^{ks_0}| \La_k f|\big\|_{p_0,\rho}
=
\|f\|_{F^{s_0}_\infty[L^{p_0,\rho}]}
\end{align*}
and \eqref{jawerthextreme} is proved. 
\end{proof}


\begin{theorem}\label{frankethm}
Suppose $0<d(1/p_0-1/p_1)=s_0-s_1$,  $0<q_0,q_1,r_0,r_1\le \infty$.
Then the embedding
$B^{s_0}_{q_0} [L^{p_0,r_0}] \hookrightarrow
F^{s_1}_{q_1} [L^{p_1,r_1}]$ holds if and only if $q_0\le r_1$.
\end{theorem}
\begin{proof} For the necessity of the condition $q_0\le r_1$ see 
\S\ref{franke-jawerth-nec}. It now suffices to prove for any $q>0$ and for $0<\rho\le \infty$,
\Be\label{frankeextreme}
B^{s_0}_\rho[L^{p_0,\infty}] \hookrightarrow
F^{s_1}_{q} [L^{p_1,\rho}].
\Ee
Assume that $f\in B^{s_0}_{\rho}[L^{p_0,\infty}]$.
Define
$$
h_k(x) = 2^{ks_1} \sum_{Q\in \fQ_k}\bbone_Q(x) \inf_Q \fM_k(\La_k f)
$$
where $\fQ_k$ denotes the grid of dyadic cubes with side length $2^{-k}$. Then
\begin{align}
\|f\|_{F^{s_1}_q(L^{p_1,\rho})}&=\Big\|\sum_{k=0}^\infty 2^{ks_1q}|\La_k f|^q \Big\|_{p_1/q,\rho/q}^{1/q}
\le \Big\| \sum_{k=0}^\infty h_k^q \Big\|_{p_1/q,\rho/q}^{1/q}
\notag
\\
&=\Big(\sup_{\|g\|_{(p_1/q)',(\rho/q)'}=1} \sum_{k=0}^\infty \int h_k(x)^q g(x) \,dx \Big)^{1/q}.
\label{duality}
\end{align}
Note that the rearrangement function of $h_k$ is constant on the intervals 
$[2^{-kd}(n-1), 2^{-kd}n)$ for $n=1,2,\dots$.
Thus for fixed $k$
\begin{align*} 
&\Big|\int h_k(x)^q g(x) dx\Big| \le \int_0^\infty (h_k^q)^*(t) g^*(t) dt 
\\
&= \int_0^\infty (h_k^*(t))^q g^*(t) dt  = \sum_{n=1}^\infty h_k^*(2^{-kd}(n-1))\int_{2^{-kd}(n-1)}^{2^{-kd}n} g^*(t) dt
\\
&\le \sum_{n=1}^\infty (h_k^*(2^{-kd}(n-1)))^q 2^{-kd} g^{**} (2^{-kd}n)
\end{align*}
We sum in $k$ and get
\begin{align*} 
&\sum_{k=0}^\infty\sum_{n=1}^\infty (h_k^*(2^{-kd}(n-1)))^q 2^{-kd} g^{**} (2^{-kd}n)
\\
&\lc \sum_{k=0}^\infty\sum_{l=0}^\infty(h_k^*(2^{ld-kd-1}))^q 2^{ld-kd} g^{**} (2^{ld-kd})
\\
&= \sum_{l=0}^\infty \sum_{k=0}^\infty 2^{(ld-kd)q/p_0}2^{k(s_0-s_1)q} h_k^*(2^{ld-kd-1})^q \times \\
&\qquad\qquad\qquad  
2^{(ld-kd)(1-q/p_0)}
 g^{**} (2^{ld-kd}) 
2^{-kd(1/p_0-1/p_1)q}
\end{align*} 
where we have used $s_0-s_1=d/p_0-d/p_1$. By H\"older's inequality
the last displayed expression is dominated by
\begin{align*} 
& \sum_{l=0}^\infty \Big( \sum_{k=0}^\infty 2^{(ld-kd)\rho/p_0}2^{k(s_0-s_1)\rho} h_k^*(2^{ld-kd-1})^\rho\Big)^{q/\rho} \times \\
&\qquad\qquad\qquad  
\Big(\sum_{k=0}^\infty \Big[ 2^{(ld-kd)(1-q/p_0)}
 g^{**} (2^{ld-kd}) 
2^{-kd(1/p_0-1/p_1)q}\Big]^{(\rho/q)'}\Big)^{1-q/\rho}
\\
& \lc \Big( \sum_{k=0}^\infty2^{k(s_0-s_1)\rho}   \sup_{m\ge 0} 2^{(md-kd-1)\rho/p_0}
h_k^*(2^{md-kd-1})^\rho\Big)^{q/\rho} 
\times \\
&\qquad\qquad\quad\sum_{l=0}^\infty 
\Big(\sum_{k=0}^\infty \Big[ 2^{(ld-kd)(1-q/p_0)}
 g^{**} (2^{ld-kd}) 
2^{-kd(1/p_0-1/p_1)q}\Big]^{(\rho/q)'}\Big)^{1-q/\rho}.
\end{align*}
Now
we have 
\begin{align*}
&\Big( \sum_{k=0}^\infty2^{k(s_0-s_1)\rho}   \sup_{m\ge 0} 2^{(md-kd-1)\rho/p_0}
h_k^*(2^{md-kd-1})^\rho\Big)^{q/\rho} 
\\&\lc 
\Big( \sum_{k=0}^\infty2^{ks_0\rho}  \| 2^{-ks_1} h_k \|_{p_0,\infty}^\rho\Big)^{q/\rho}
\lc 
\Big( \sum_{k=0}^\infty2^{ks_0\rho}  \| \fM_k (\La_k f) \|_{p_0,\infty}^\rho\Big)^{q/\rho}
 \\&\lc 
\Big( \sum_{k=0}^\infty2^{ks_0\rho}  \| \La_k f \|_{p_0,\infty}^\rho\Big)^{q/\rho} = 
\|f\|_{B^{s_0}_\rho[L^{p_0,\infty}]}^q.
\end{align*}
Finally we estimate,
summing $\sum_{l=0}^\infty 2^{-ld(1/p_0-1/p_1)q}\lc 1$, 
\begin{align*}&\sum_{l=0}^\infty 
\Big(\sum_{k=0}^\infty \Big[ 2^{(ld-kd)(1-q/p_0)}
 g^{**} (2^{ld-kd}) 
2^{-kd(1/p_0-1/p_1)q}\Big]^{(\rho/q)'}\Big)^{1-q/\rho}
\\ &\lc 
\sup_{l\ge 0} \Big(\sum_{k=0}^\infty \Big[ 2^{(ld-kd)(1-q/p_0)}
 g^{**} (2^{ld-kd}) 
2^{(ld-kd)(1/p_0-1/p_1)q}\Big]^{(\rho/q)'}\Big)^{1-q/\rho}
\\
&=\sup_{l\ge 0} \Big( \sum_{k=0}^\infty \big[ 2^{(kd-ld)(1-q/p_1)} g^{**} (2^{ld-kd}) \big] ^{(\rho/q)'} 
\Big)^{1-\rho/q} \lc \|g\|_{(p_1/q)' ,(\rho/q)'} \lc 1.
\end{align*}
 Going back to \eqref{duality} we see that 
 $\|f\|_{F^{s_1}_q[L^{p_1,\rho}]} \lc \|f\|_{B^{s_0}_\rho[L^{p_0,\infty}]}$ and the proof of \eqref{frankeextreme}  is complete.
 \end{proof}


\section{Conclusion of the proofs of  Theorems \ref{BinFthm}, \ref{FinBthm},  \ref{BinBthm},  \ref{FinFthm}}\label{conclusion}

The necessity of all conditions was shown in \S\ref{meassmoothness}. 
The proofs of  the embeddings in parts  (iv)-(vi) of  Theorem \ref{BinFthm} and Theorem \ref{FinBthm} were shown in  \S\ref{vectvalemb} (see Corollaries \ref{btlthm} and \ref{btldualthm}).  The embeddings in part (iii)  of Theorem \ref{BinFthm}  
and  Theorem \ref{FinBthm}  are  covered by 
Theorem \ref{frankethm}   and Theorem \ref{jawerththm}, respectively.

We consider part (ii) of Theorem \ref{BinFthm}. Let $s_0>s_1$, $r_0\le r_1$, $p_0=p_1=p$.
 Let $\eps>0$ such that $s_0-\eps>s_1$, and let $v<\min \{q_0, p_1, q_1,r_1\}$. By parts (iv) or (v)
of Theorem \ref{BinFthm} we have $B_v^{s_0-\eps}[L^{p,r_0}]\hookrightarrow F^{s_0-\eps}_{q_1}[L^{p,r_1}]$ 
and (ii) follows if we combine this with the trivial embeddings
$B^{s_0}_{q_0}[L^{p,r_0}]\hookrightarrow 
B^{s_0-\eps}_{v}[L^{p,r_0}]$ and 
$F^{s_0-\eps}_{q_1}[L^{p,r_1}]\hookrightarrow 
F^{s_1}_{q_1}[L^{p,r_1}]$. 

The proof of part (ii) of Theorem \ref{FinBthm} is similar. Moreover if we use part (iii) in Theorems 
\ref{BinFthm} and \ref{FinBthm} the proofs of part (i) in those theorems  follows the same pattern as above.

Finally we consider Theorems \ref{BinBthm} and \ref{FinFthm}.
Part (iv) of these theorems are proved  by using  embeddings of $L^{p,r}$ spaces and of $\ell^q$ spaces.

To see part (iii) of Theorem \ref{BinBthm}, assume $q_0\le q_1$ and let $\widetilde p$ and $\widetilde s$ be such that
$p_0<\widetilde p<p_1$, $s_1<\widetilde s<s_0$ and 
$\widetilde s-s_1=d/\widetilde p-d/p_1$ and thus $s_0-\widetilde s= d/p_0-d/\widetilde p$.
Pick $\widetilde r$ such that $q_0\le \widetilde r\le q_1$ and then 
$$B^{s_0}_{q_0}[L^{p_0,r_0}] \hookrightarrow F^{\widetilde s}_{\widetilde q}[L^{\widetilde p,\widetilde r}] \hookrightarrow B^{s_1}_{q_1}[L^{p_1,r_1}]
$$
for arbitrary $\widetilde q$, 
by 
Theorem \ref{frankethm} for the first embedding and Theorem \ref{jawerththm} for the second.

To see part (iii) of Theorem \ref{FinFthm} assume $r_0\le r_1$. Pick $\widetilde q$ such that $r_0\le \widetilde q\le r_1$ and then
$$F^{s_0}_{q_0}[L^{p_0,r_0}] \hookrightarrow B^{\widetilde s}_{\widetilde q}[L^{\widetilde p,\widetilde r}] \hookrightarrow F^{s_1}_{q_1}[L^{p_1,r_1}]
$$
for arbitrary $\widetilde r$, 
by 
Theorem \ref{jawerththm} for the first embedding and Theorem \ref{frankethm} for the second. 
 
 Given parts (iv), (iii) of Theorems \ref{BinBthm} and \ref{FinFthm} the parts (i), (ii) in the noncritical ranges can be  obtained by the argument  above.

\appendix

\section{Remarks on Mikhlin-H\"ormander multipliers} \label{Hoersect}
Part (iii) of Theorems 
 \ref{BinFthm}  and  \ref{FinBthm} 
(i.e. Theorems 
  \ref{frankethm} and \ref{jawerththm})
 can be applied to clarify the connection  between  certain sharp versions of 
the Mikhlin-H\"ormander multiplier theorem (\cite{hoer}). Set $T_mf=\cF^{-1}[m\widehat f]$.  
Let $\varphi$ be a nontrivial  radial smooth functions which is compactly supported in $\bbR^d\setminus \{0\}$.
We first recall the endpoint bound
\Be \label{see-a}\|T_m\|_{L^p\to L^{p,2}}  
\lc \sup_{t>0} \| \varphi m(t\cdot) \|_{B^{s}_1 [L^{d/s}]}, \quad s=d(1/p-1/2), \,\,1<p\le 2,
\Ee 
which was proved by one of the authors  in \cite{seeger1}. Moreover one gets $H^1\to L^{1,2}$ boundedness under the condition  
$\sup_{t>0} \| \varphi m(t\cdot) \|_{B^{d/2}_1 [L^2]}<\infty$,
see \cite{seeger2}.  Note that \eqref{see-a} immediately  implies that
\begin{subequations} 
\Be\label{see-b}\|T_m\|_{L^p\to L^{p}} 
\lc \sup_{t>0} \| \varphi m(t\cdot) \|_{B^{s}_1 [L^{d/s}]}, \quad d|1/p-1/2|<s<d.
\Ee
Indeed, by the standard Sobolev imbedding theorem for Besov spaces we may assume that $s<d/2$. Define $p_0$ by $d(1/p_0-1/2)=s$, so that $1<p_0<p<2$. Then  \eqref{see-a} gives $L^{p_0}\to L^{p_0,2}$ boundedness, and by the Marcinkiewicz interpolation theorem, and a subsequent duality argument we get \eqref{see-b}.


A  recent
result  by Grafakos and Slav\'ikov\'a \cite{grafakos-slavikova} 
states that for $1<p<\infty$ 
\begin{align}\label{grsl}\|T_m\|_{L^p\to L^p} 
&\lc \sup_{t>0} \| (I-\Delta)^{s/2}[\varphi m(t\cdot)] \|_{L^{d/s,1}}, 
\\
&\approx \sup_{t>0}\|\varphi m(t\cdot)\|_{F^s_2(L^{d/s,1})}, 
\quad d|1/p-1/2|<s<d, \nonumber
\end{align}
\cf. \eqref{besselpot}. 
\end{subequations}
For fixed $s$ the  relation between the norms on the right hand side in \eqref{see-b} and \eqref{grsl} is not immediately clear. The spaces  
$F^s_2[L^{d/s,1}]$ and $B^s_1[L^{d/s}]$ are not comparable;  we have 
$F^s_2[L^{d/s,1}]
\nsubseteq B^s_1[L^{d/s}]$ by \S\ref{q0q1necsect},  and we get 
 $B^s_1[L^{d/s}] \nsubseteq F^s_2[L^{d/s,1}]$ by the necessity of the condition $r_0\le r_1$ in \S\ref{p0lep1}.
However we do have the embeddings
\Be \label{applH}
B^{s_3}_1[L^{d/s_3}]
\hookrightarrow
F^{s_2}_q[L^{d/s_2, 1}]
\hookrightarrow
B^{s_1}_1[L^{d/s_1}], \quad s_1<s_2<s_3,\,\, q>0.
\Ee
The first inclusion  in \eqref{applH}  follows from Theorem \ref{BinFthm} (iii) and the second from Theorem \ref{FinBthm} (iii). Since both 
statements \eqref{see-b}, \eqref{grsl} involve the same  open $s$-interval we may apply  \eqref{applH} for $s_1>d|1/p-1/2|$ and $q=2$ to see  
   that they 
  cover $L^p$ boundedness for exactly the same set  of multiplier transformations.

\medskip 

\noindent{\it Further directions and open problems.}
\subsection{\it} 
 Is  there  an endpoint inequality such as  \eqref{see-a} in terms of localized
$F^{s}_2[L^{d/s,1}]$ spaces, when $s=d|1/p-1/2|$?

\subsection{\it}
It was proved in \cite{seeger3}  that
\Be\label{seeger3result}\|T_m\|_{L^p\to L^{p}} \lc \sup_{t>0} \| \varphi m(t\cdot) \|_{B^{d|1/p-1/2|}_1 [L^{q}]}\Ee for 
$1/q>|1/p-1/2|$,  $1<p<\infty$ (see also  \cite{bs}, \cite{seeger3} for corresponding results on Hardy spaces).  In the first   version of our  paper posted as arXiv:1801.10570 in January 2018  we raised the question  whether
 the space $B^{d|1/p-1/2|}_1 [L^{q}]$ in this result  can be replaced with the Lorentz-Sobolev space 
$H^{d|1/p-1/2|}_{(q,r)}\equiv F^{d|1/p-1/2|}_2 [L^{q,r}]$.  A negative answer to this question was  recently given by  Slav\'ikov\'a \cite{slavikova}.
However the following question remains interesting.

\subsubsection{\it Question} \label{openqu} Let $1/q>|1/p-1/2|$. 
For which $u>1$, if any, can one replace $B^{d|1/p-1/2|}_1 [L^{q}]$ 
in  inequality \eqref{seeger3result} 
 with $B^{d|1/p-1/2|}_u [L^{q}]$?
 
 The parameter range $1<u\le p$ seems of particular interest, as will be discussed in the following subsection.

\subsection{\it }\label{moreonconj} 
Let $1\le p\le 2$ and let $M^p$ be the usual Fourier multiplier space. We  note the following obvious chain of inequalities
$$\sup_{t>0} \|\cF^{-1}[\varphi m(t\cdot)]\|_{L^p} 
\lc 
\sup_{t>0} \|\varphi m(t\cdot)\|_{M^p} \lc  \|m\|_{M^p}.$$
The corresponding inclusions for the spaces  defined by the above norms are known to be strict.
It is of interest to identify function spaces  for which
$\sup_{t>0} \|\cF^{-1}[\varphi m(t\cdot)]\|_{L^p} $ is finite and the corresponding multiplier question is open. As an example one has  the inequality
\Be \label{locBernsteinvariant}
\sup_{t>0} \|\cF^{-1} [\varphi m(t\cdot)]\|_{L^p} \lc 
\sup_{t>0} \|\varphi m(t\cdot)\|_{B^{d(1/p-1/2)}_p(L^{2})}
\Ee 
by using Bernstein's inequality. In view of the compact support of $\varphi$ we can  replace 
$B^{d(1/p-1/2)}_p(L^{2})$ with 
 $B^{d(1/p-1/2)}_p(L^{q,r})$ for $2<q<\infty$, $0<r\le \infty$.

We observe that a slight modification of the above mentioned example  by Slav\'ikov\'a  \cite{slavikova} exhibits a sequence of functions  $g_K$  which is bounded 
in $F^{d(1/p-1/2)}_2(L^{q,r})$, with $2\le q<\infty$, $r>0$, for which 
$\sup_{t>0}  \|\cF^{-1}[\varphi g_K(t\cdot)]\|_{L^p} $ is unbounded. Indeed one  may choose this  sequence  to be  supported in $\{\xi: 1/4<|\xi|<  2\}$.

To define  $g_K$   let 
$\Psi$ be any  nontrivial $ C^\infty_c(\bbR^d)$ function  
supported on a ball of radius $1/10$ centered at the origin of  $\bbR^d$. 
Let $K$, $N$ be large integers let $(k,N)\mapsto \nu(k,N)$ be an enumeration of $\bbN^d\times\bbN$ and denote by $r_\nu$ the Rademacher functions.
Let, for $0\le s< \floor {M+1}$ and $t\in [0,1]$, 
\begin{multline*}g_{s, K}(\xi,t)
= \sum_{K/2<N\le K} (K2^N)^{-s} h_{K,N}(\xi,t)\\
\text{where }  h_{K,N}(\xi,t)=
\sum_{\substack {\ka\in \bbN^d:\\ N<2^{-N}|\ka|<N+1/2}} r_{\nu(k,N)}(t) \Psi(2^NK (\xi - 2^{-N} K^{-1}\ka)).
\end{multline*}
Then $\xi\mapsto g_s(\xi,t)$ is supported in the annulus $\sA_1=\{1/4<|\xi|\le 1\}$.
The calculation in \cite{slavikova} (using Khinchine's inequality) shows that 
\Be\label{sl-lowerbd}
\Big(\int_0^1 \big\|\cF^{-1}[ g_{s,K}(\cdot,t)]\big\|_p^p dt\Big)^{1/p} \gc_s  K^{1/p-1/2}, \text{ when } s=d(\frac 1p-\frac 12).
\Ee
Moreover, 
\Be
\sup_{t\in [0,1]} \|g_{s,K} (\cdot,t) \|_{F^s_2(L^{q,r})} \le C(s),  \quad 1< q<\infty, \, r>0,
\Ee
with polynomial growth in $s$.
This can be verified by direct calculation as in \cite{slavikova}. Alternatively,  one can immediately check this inequality for $r=q$ and $s=0,1,2, \dots$ by using the standard Sobolev norms, and then apply an analytic interpolation argument to get the same statement for all $s\ge 0$. Fixing $s$,  one can then apply real interpolation (with varying $q$) to get the statement for  all $r>0$, $s\ge 0$.

In view of question \ref{openqu} it is instructive to replace 
$F^s_2(L^{q,r})$ by $B^s_u(L^{q,r})$ for some $q\ge 2$, $u\le q$.
We shall now impose  the assumption  that the generating function $\Psi$ satisfies 
$\int\Psi(y) P(y) dy=0$ for all polynomials of degree at most $M$, for some $M> s+2d$. It is easy to check that the argument \cite{slavikova}  for the lower bound \eqref{sl-lowerbd} still goes through. We also have  
\Be\label{BesovgKupper}
\sup_{t\in [0,1]} \|g_{s,K}(\cdot,t) \|_{B^s_u(L^{q,r}) }
\lc K^{1/u-1/q}, \quad 2\le q<\infty, \, r>0,\, u\le q.
\Ee 
This is verified by using the Littlewood-Paley type characterization  using compactly supported $\psi_k$  in \eqref{localmeanschar}, imposing the condition that $\psi_k$ for $k>0$ have a large number of  vanishing moments, say more than $s+d+A$ for large $A$.
One  verifies that  for $k\ge 0$
$$\|\psi_k* h_{K,N}\|_{L^{q,r}} \lc K^{-1/q} \min\{ 2^N K 2^{-k}, (2^N K 2^{-k})^{-1}
\}^{-1-s}, $$
first for $q=r$ and then for general $r$ by real interpolation. This can be used to check
\eqref{BesovgKupper}. 

Inequalities   \eqref{sl-lowerbd}  and \eqref{BesovgKupper} 
show that the inequality 
\Be \label{Bernsteinvariant}
\|\cF^{-1} g\|_{L^p} \lc \|g\|_{B^{d(1/p-1/2)}_u(L^{q,r})}
\Ee 
may fail for compactly supported $g$
when  $u\le q$, $1/u-1/q<1/p-1/2$,  $r>0$.  
By \eqref
{locBernsteinvariant}
and its version with 
 $B^{d(1/p-1/2)}_p(L^{q,r})$ for $q>2$ the range $1<u\le p$ is of particular interest in question \ref{openqu}.

\section{On the Constant in the \\Triangle Inequality for $L^{p,r}$, $p<1$}\label{appendix}
In what follows we work with the quasinorm $\|\cdot\|_{p,r} $ on $L^{p,r}$ as defined in
\eqref{qn} or \eqref{Lorentznorms}. The following result was referenced in \S\ref{reviewsect}, together with an open question.

\begin{proposition}\label{ptrianglelemlor}  
Let $0<p<1$, $p<r<\infty$. Then 
\[
\Big\|\sum_k f_k \Big\|_{{p,r}} \le C(p,r) \Big(\sum_k \|f_k\|_{{p,r}}^p \Big)^{1/p}, 
\]
where 
\Be\label{constants}C(p,r) \le A^{1/p}  \Big( \frac 1{1-p}\Big )^{1/p-1/r} \Big(1 + \frac pr \log \frac{1}{1-p} \Big)^{1/p-1/r}
\Ee
and $A$ does not depend on $p$ and $r$.
 \end{proposition}

We shall need the following lemma. It can be used to prove the inequality \eqref{embineq} 
 when applied in combination with  \eqref{Lorentznorms}.

\begin{lemma}\label{embeddinglemma}
Let $g:\bbR^+\to \bbR$ be a Riemann integrable function
and let $r<q$.
Then for $0\le \alpha< \beta \le  \infty$, $0<p<\infty$,
\[\sup_{\alpha\le\sigma\le \beta} \sigma|g(\sigma)|^{1/p}\le 
\Big( q \int_\alpha^\beta \sigma^{q} |g(\sigma)|^{q/p} \frac {d\sigma}{\sigma} \Big)^{1/q}
\le
\Big( r \int_\alpha^\beta \sigma^{r} |g(\sigma)|^{r/p} \frac {d\sigma}{\sigma} \Big)^{1/r}.
\]
\end{lemma}

\begin{proof} 
We prove the second inequality, as the first one follows by letting $q\to\infty$.
We may assume that $g$ is a  nonnegative step function
on $[\alpha,\beta]$, i.e. there is a partition $\alpha=b_0<b_1<\dots<b_N=\beta$ so that
$g(s)=c_j$ if $b_{j-1}<s<b_j$; here $c_j\ge 0$. The inequality then reduces to
\[\Big(\sum_{j=1}^N c_j^{ q/p}(b_j^q-b_{j-1}^q)\Big)^{1/q}
\le
 \Big(\sum_{j=1}^N c_j^{r/p}(b_j^r-b_{j-1}^r)\Big)^{1/r}
\]
We set $v_j=c_j^{ r/p}$, $a_j=b_j^r$, and $s=q/r$ so that $s>1$, and see that the last inequality follows from
\Be\label{normalization}
\Big(\sum_{j=1}^N v_j^{s}(a_j^s-a_{j-1}^s)
\Big)^{1/s}
\le
 \sum_{j=1}^N v_j(a_j-a_{j-1}).
\Ee
Since $s\ge 1$ we may  (by the triangle inequality for the $s$-norms) 
replace
$(a_j^s-a_{j-1}^s)$ on the left hand side of \eqref{normalization} with $(a_j-a_{j-1})^s$, and \eqref{normalization} follows from  $\|\cdot\|_{\ell^s}\le \|\cdot\|_{\ell^1}$.
\end{proof}
%

\begin{proof}[Proof of  Proposition \ref{ptrianglelemlor}]
The proof is based on ideas in \cite{st-nw}, \cite{sttw}. For given $\alpha>0$ we split $f_k= g_{k,\alpha}+ b_{k,\alpha}$ where
\[g_{k,\alpha}(x) = \begin{cases} f_{k}(x) &\text{ if } |f_k(x)|\le \alpha
\\
0 &\text{if }  |f_k(x)|> \alpha\end{cases}
\]
and let $b_{k,\alpha}=f_k-g_{k, \alpha}$.
Let 
$$E_\alpha= \{x: b_{k,\alpha}(x) \neq 0 \text{ for some $k$}\}.$$
Then 
\[ \mu(\{x: |\sum_k f_k(x)|>\alpha\}) \le \mu(E_\alpha) + \mu(\{x:\sum_k|g_{k,\alpha}(x)|>\alpha\})
\] and therefore
\begin{multline}\label{decomp}
\Big\|\sum_k f_k\Big\|^p_{{p,r}} \le 
 \Big(r \int_0^\infty \alpha^r \mu(E_\alpha)^{r/p} \frac{d\alpha}{\alpha}\Big)^{p/r}
\\+ 
\Big( r\int_0^\infty \alpha^r \Big[\mu\big(\{x:\sum_k|g_{k,\alpha}(x)|>\alpha\}\big)\Big]^{r/p}\frac{d\alpha}{\alpha}\Big)^{p/r}.
\end{multline}
Now 
\[\mu(E_\alpha)\le \sum_{k} \mu (\{x: b_{k,\alpha}(x) \neq 0\})
\le \sum_{k}\mu\ci{f_k}(\alpha)
\]
and hence
\begin{align}
&\Big(r \int_0^\infty \alpha^r \mu(E_\alpha)^{r/p}\frac{d\alpha}{\alpha}\Big)^{p/r} 
\le 
\Big(r \int_0^\infty \alpha^r \Big(\sum_k \mu\ci{f_k}(\alpha)\Big)^{r/p}\frac{d\alpha}{\alpha}\Big)^{p/r} 
\notag
\\
\label{Ealphaest}
&\le \sum_{k} \Big(r \int_0^\infty \alpha^r  [\mu\ci {f_k}(\alpha)]^{r/p}\frac{d\alpha}{\alpha}\Big)^{p/r} 
 \le \sum_k \|f_k\|_{{p,r}}^p 
\end{align}
here we have used Minkowski's inequality in $L^{r/p}$ (and thus our assumption $r\ge p$).

We now further decompose $g_{k,\alpha}= l_{k, \alpha} +m_{k,\alpha}$ into a low and a middle part where for a suitable constant $B>1$,
\[
l_{k,\alpha}(x)= \begin{cases}
f_k(x) &\text {if } |f_k(x)|\le \alpha/B 
\\
0 &\text{ if } |f_k(x)|>\alpha/B
\end{cases}
\]

and
\[
m_{k,\alpha}(x)= \begin{cases}
0 &\text {if } |f_k(x)|\le \alpha/B 
\\
f_k(x)
&\text{ if }\alpha/B<|f_k(x)|\le \alpha 
\\
0 &\text{ if } |f_k(x)|>\alpha
\end{cases}
\]
For a  favorable choice for $B$ see \eqref{Bchoice} below.
Now
\begin{multline*}
\mu(\{x:\sum_k|g_{k,\alpha}(x)|>\alpha\})
\\ \le
\mu(\{x:\sum_k|l_{k,\alpha}(x)|>\alpha/2\})
+
\mu(\{x:\sum_k|m_{k,\alpha}(x)|>\alpha/2\})
\end{multline*}
and therefore by Minkowski's inequality and a subsequent change of variable
\begin{subequations}
\begin{align}
&\Big(r \int_0^\infty \alpha^r \Big[\mu(\{x:\sum_k|g_{k,\alpha}(x)|>\alpha\})\Big]^{r/p} \frac{d\alpha}{\alpha}\Big)^{p/r}\notag
\\ &\le 2^p
 \Big(r \int_0^\infty \alpha^r \Big[\mu(\{x:\sum_k|l_{k,2\alpha}(x)|>\alpha\})\Big]^{r/p} \frac{d\alpha}{\alpha}\Big)^{p/r}
 \label{lkterms}
\\& +2^p\Big(r \int_0^\infty \alpha^r \Big[\mu(\{x:\sum_k|m_{k,2\alpha}(x)|>\alpha\})\Big]^{r/p} \frac{d\alpha}{\alpha}\Big)^{p/r}.
\label{mkterms}
 \end{align} 
\end{subequations}

Next, by Tshebyshev's inequality
\begin{align*}
&\Big(r \int_0^\infty \alpha^r \Big[\mu(\{x:\sum_k|l_{k,2\alpha}(x)|>\alpha\})\Big]^{r/p} \frac{d\alpha}{\alpha}\Big)^{p/r}
\\
&\le \Big(r \int_0^\infty \alpha^r\Big [\frac {1}{\alpha}\int \sum_k|l_{k,2\alpha}| d\mu
 \Big]^{r/p} \frac{d\alpha}{\alpha}\Big)^{p/r}
 \\
 &\le \Big(r \int_0^\infty \Big [\sum_k\alpha^p\,\frac {1}{\alpha}\int_0^{2B^{-1}\alpha} \mu\ci{l_{k,2\alpha}}(\beta)d\beta
 \Big]^{r/p} \frac{d\alpha}{\alpha}\Big)^{p/r};
 \end{align*} 
  here we have used 
 $\mu_{l_{k,2\alpha}}(\beta)=0$ when $\beta>B^{-1}2\alpha$.
 We now use Minkowski's inequality in $L^{r/p}(d\alpha/\alpha)$. 
 Since $|l_{k,2\alpha}|\le |f_k|$ we see that the last displayed expression is dominated by
 \begin{align*}
 & \sum_k\Big(r \int_0^\infty \Big [\alpha^p \frac {1}{\alpha}\int_0^{2B^{-1}\alpha} \mu\ci{f_k}(\beta)d\beta
 \Big]^{r/p} \frac{d\alpha}{\alpha}\Big)^{p/r}\\
 &=\sum_k 
 \Big(r \int_0^\infty \alpha^r \Big[ \int_0^{2B^{-1}} \mu\ci{f_k}(s\alpha)ds 
 \Big]^{r/p} \frac{d\alpha}{\alpha}\Big)^{p/r}.
 \end{align*}
 We continue as in the proof of Hardy's inequalities and estimate using the integral Minkowski inequality
 \begin{align*}
 &\Big(r \int_0^\infty \alpha^r \Big[ \int_0^{2B^{-1}} \mu\ci{f_k}(s\alpha)ds 
 \Big]^{r/p} \frac{d\alpha}{\alpha}\Big)^{p/r}
 \\ &\le
 \int_{0}^{2B^{-1}} s^{-p}\Big( \int_0^\infty r\beta^{r-1} \mu\ci{f_k}(\beta)^{r/p} d\beta\Big)^{p/r} 
= \frac{2^{1-p}B^{p-1}}{1-p}  \|f_k\|_{L^{p,r}}^p
 \end{align*}

Thus, combining estimates we get
\Be\eqref{lkterms}\,\le \,\frac{2}{1-p}  B^{p-1}\sum_k \|f_k\|_{{p,r}}^p.
\label{lkest}
\Ee

We now estimate the terms in \eqref{mkterms}. We write $[0,\infty)$ as a union over the intervals $I_n=[B^n,B^{n+1}]$, $n\in \bbZ$ and apply Lemma \ref{embeddinglemma} to each interval.
Then \eqref{mkterms} is estimated by
\begin{align}
&2^p\Big(\sum_{n\in \bbZ}r\int_{B^n}^{B^{n+1}} \alpha^r \Big[\mu(\{x:\sum_k|m_{k,2\alpha}(x)|>\alpha\})\Big]^{r/p} \frac{d\alpha}{\alpha}\Big)^{p/r}
\notag
\\
&\le
2^p\Big( \sum_{n\in \bbZ}\Big[p\int_{B^n}^{B^{n+1}} \alpha^p \mu(\{x:\sum_k|m_{k,2\alpha}(x)|>\alpha\}) \frac{d\alpha}{\alpha}\Big]^{r/p}\Big)^{p/r}.
\notag
\end{align}
We now define
\[f_{k,n}(x)= 
\begin{cases} f_k(x) &\text{ if } B^n \le| f_k(x)|\le B^{n+2}
\\ 0 &\text{ otherwise}
\end{cases}
\]
and observe
\[|m_{k, 2\alpha}(x)| \le |f_{k,n}(x)| \text{ if } B^n\le \alpha \le B^{n+1}.\]
Hence we get 
\begin{align*}
&\Big(r\int_{0}^{\infty} \alpha^r \Big[\mu(\{x:\sum_k|m_{k,2\alpha}(x)|>\alpha\})\Big]^{r/p} \frac{d\alpha}{\alpha}\Big)^{p/r}
\\
&\le \Big(\sum_{n=-\infty}^\infty 
\Big[\int_{0}^{\infty} p\alpha^{p-1} \mu(\{x:\sum_k|f_{k,n}(x)|>\alpha\}) d\alpha\Big]^{r/p}\Big)^{p/r}
\\
&\le \Big(\sum_{n=-\infty}^\infty \Big\|\sum_k |f_{k,n}| \Big\|_p^r \Big)^{p/r}
\,\le \Big(\sum_{n=-\infty}^\infty \Big[ \sum_k \|f_{k,n}\|_p^p\Big]^{r/p}  \Big)^{p/r}
\\
&\le \sum_k \Big(\sum_{n=-\infty}^\infty \|f_{k,n}\|_p^r  \Big)^{p/r}.
\end{align*}
Here we have used the triangle inequality in $L^p$, $p<1$ and Minkowski's inequality for the sequence space 
$\ell^{r/p}$, $r\ge p$. Now
\begin{align*}
&\|f_{k,n}\|_p^p\le 
\int_0^{B^n}p\alpha^{p-1} \mu_{f_k}(B) d\alpha+ 
\int_{B^n}^{B^{n+2}} p\alpha^{p} \mu_{f_k}(\alpha)\frac{ d\alpha}\alpha
\\
&\le B^{np} \mu_{f_k}(B) + (\log B^2)^{1-p/r} 
\frac{p}{r^{p/r}}
\Big(\int_{B^n}^{B^{n+2}} r\alpha^{r} \mu_{f_k}(\alpha)^{r/p} \frac{d\alpha}{\alpha}\Big)^{p/r},
\end{align*}
by H\"older's inequality, and
\begin{align*}
&\Big(\sum_n \|f_{k,n}\|_p^r\Big)^{p/r}
\le \Big(\sum_n B^{nr} \mu_{f_k}(B)^{r/p} \Big)^{p/r}\\
&\qquad\qquad\qquad \qquad+
(\log B^2)^{1-p/r} 
\frac{p}{r^{p/r}}\Big(\sum_n\int_{B^n}^{B^{n+2}}r\alpha^{r-1} \mu_{f_k}(\alpha)^{r/p}d\alpha\Big)^{r/p}
\\
&\le \Big(1+\frac{p}{r^{p/r}}2^{p/r}
(\log B^2)^{1-p/r}  \Big) 
\Big(\int_0^\infty r\alpha^{r-1}\mu_{f_k}(\alpha) d\alpha\Big)^{p/r}.
\end{align*} Consequently
\Be
\eqref{mkterms}\,\le \, 2^p\Big(1+\frac{p}{r^{p/r}}2^{p/r}
(2\log B)^{1-p/r}  \Big) \sum_k \|f_k\|_{L^{p,r}}^p\,.
\label{mkest}
\Ee
We now  make the choice
\Be\label{Bchoice}
B= (1-p)^{-\frac p{r(1-p)} }
\Ee
so that
 \[
 \frac{B^{p-1}}{1-p}=(1-p)^{-(1-p/r)}, \qquad  \log B= \frac{p}{r(1-p)} \log \frac{1}{p-1}\,.
 \]
Combining the various estimates we obtain
\begin{multline}
\Big\|\sum_k f_k \Big\|^p_{L^{p,r}} \\ \le
\Big( 1+ \frac 2{(1-p)^{1-p/r}} + \frac{p2^{p+1}}{r^{p/r}} \frac p{r(1-p) } \log(\frac 1{1-p})^{1-p/r}\Big)
\sum_k\|f_k\|_{L^{p,r}}^p
\end{multline}
which yields the lemma.
\end{proof}

\end{document}